\documentclass[reqno]{amsart}
\usepackage{amsfonts, amsmath} 
\usepackage{amssymb}
\usepackage{amsthm}
\usepackage{mathrsfs}
\usepackage[english]{babel}
\usepackage[latin1]{inputenc}
\usepackage[T1]{fontenc}
\usepackage[top=3.3cm, bottom=3.3cm, left=3.3cm, right=3.3cm]{geometry}
\usepackage{color}
\usepackage{accents}
\usepackage[mathscr]{euscript}

\begin{document}

		\newtheorem{definition}{Definition}[section]
		\newtheorem{theorem}{Theorem}[section]
		\newtheorem{lemma}{Lemma}[section]
		\newtheorem{corollary}{Corollary}[]
		\newtheorem*{T. lemma}{Technical lemma }
		\newtheorem{proposition}{Proposition}[section]
		\numberwithin{equation}{section}
		\newtheorem{remark}{Remark}[section]

		\title[]{
		Quasilinear elliptic problem without Ambrosetti-Rabinowitz condition involving a potential in Musielak-Sobolev spaces setting   
		}
		\author{Soufiane MaatouK}
		\email{sf.maatouk@gmail.com}
		\author{Abderrahmane El Hachimi}
		\email{aelhachi@yahoo.fr}
		\address{Center of Mathematical Research and Applications of Rabat (CeReMAR), Laboratory of Mathematical Analysis and Applications (LAMA), Department of Mathematics, Faculty of Sciences, Mohammed V University, P.O. Box 1014, Rabat, Morocco.}	
		\maketitle
\begin{abstract}
	In this paper, we consider the following  quasilinear elliptic problem with potential
	$$(P)
	\begin{cases}
	-\mbox{div}(\phi(x,|\nabla u|)\nabla u)+ V(x)|u|^{q(x)-2}u= f(x,u) & \ \ \mbox{ in }\Omega,\\
	u=0 & \ \ \mbox{ on }   \partial\Omega,
	\end{cases}$$
	where $\Omega$ is a smooth bounded   domain in $\mathbb{R}^{N}$ ($N\geq 2$), $V$ is a given function in a generalized Lebesgue space $L^{s(x)}(\Omega)$, and $f(x,u)$ is a Carathéodory function satisfying suitable growth conditions. Using  variational arguments, we study the existence of weak solutions for $(P)$ in the framework of  Musielak-Sobolev spaces.  The main difficulty here is that the nonlinearity $f(x,u)$ considered does not satisfy the well-known Ambrosetti-Rabinowitz condition.
\end{abstract}
\section{Introduction}
Let $\Omega\subset \mathbb{R}^{N} (N\geq 2)$ be a bounded smooth domain. Assume that  $\phi : \Omega \times [0,+\infty) \to [0,+\infty)$ is a Carath\'{e}odory function such that for all $x\in \Omega$, we have
$$(\phi)\begin{cases}
\phi(x,0)=0, \quad \phi(x,t).t \; \mbox{ is strictly  increasing}, &\\ 
 \phi(x,t).t>0, \; \forall t>0 \; \mbox{ and }\; \phi(x,t).t\to +\infty \mbox{ as } t \to +\infty.
\end{cases}$$
In this paper, we study   the following quasilinear elliptic  problem
$$(P)
\begin{cases}
-\mbox{div}(\phi(x,|\nabla u|)\nabla u)+ V(x)|u|^{q(x)-2}u= f(x,u) & \ \ \mbox{ in }\Omega,\\
u=0 & \ \ \mbox{ on }   \partial\Omega,
\end{cases}$$
where $V$ is a potential  belonging to $L^{s(x)}(\Omega)$,  $q$ and $s$ : $\bar{\Omega} \to (1,\infty)$ are  continuous functions and $f : \Omega \times \mathbb{R} \to \mathbb{R} $ is a Carath\'{e}odory function which  satisfies some suitable growth conditions. Precise conditions concerning the functions $q$, $s, f $ and    $V$ will be given hereafter.

 Problem $(P)$ appears in many branches of mathematical physics and has been studied extensively in recent years. From an application point of
view, this problem has its backgrounds in such hot topics as image processing, nonlinear
electrorheological fluids and elastic mechanics. We refer the readers to \cite{Chen,Ruzicka} and the references therein for more background of applications. 
 In particular,  when  $\phi(x,t)=t^{p(x)-2}$, where $p$ is a continuous function on $\overline{\Omega}$ with the condition  $\min_{x\in\bar{\Omega}}p(x)>1$, the operator involved in $(P)$ is the $p(x)$-Laplacian operator, i.e. $\Delta_{p(x)}u:=\mbox{div}(|\nabla u|^{p(x)-2}\nabla u)$. This differential operator is a natural
generalization of the $p$-Laplacian operator $\Delta_{p}u:=\mbox{div}(|\nabla u|^{p-2}\nabla u)$ where $p > 1$ is
a real constant. Note that   the $p(x)$-Laplacian operator possesses more complicated
nonlinearities   than the $p$-Laplacian operator (for example, it is nonhomogeneous), so more  complicated
analysis has to be carefully carried out.

 The interest in analyzing this kind of problems   is also motivated by some recent advances in
the study of problems involving nonhomogeneous operators in   divergence form. We refer for instance to the results in \cite{Abdou,Silva,Bonanno 1,Bonanno 2,Bonanno 3,Bonanno 4,X. Fan2,Bin,V,V2,P. Zhao 2}.  The studies for $p(x)$-Laplacian problems have been extensively considered by many researchers in various ways (see e.g. \cite{Abdou,X. Fan 1,Kefi,V2}). It should  be noted that our problem $(P)$ enables the presence of many other operators such as double-phase  and variable exponent double-phase operators.

  Before moving forward, we give a review of some results related to  our work.
We start by the case where the potential $V\equiv 0$ on $\Omega$.  Fan and Zhang  in \cite{X. Fan 1}, proved the existence of a nontrivial solution and obtained infinitely many solutions   for  a Dirichlet problem involving the $p(x)$-Laplacian operator.  Mihailescu,  Pucci and  V. Radulescu in \cite{Borenau}, have studied the question of multiplicity of solutions for a class of anisotropic elliptic equations with variable exponent. In the anisotropic setting we also refer the reader to \cite{Pucci1,Pucci2}.  Cl\'{e}ment, Garc\'{\i}a-Huidobro and  Schmitt  in \cite{Clément}, established the  existence of a nontrivial solution  for more general quasilinear equation in the framework of Orlicz-Sobolev spaces, in the case where the function $\phi$ considered in $(P)$ is independent of $x$, i.e.  $\phi(x,t)=\phi(t)$. Remaining in  the  Orlicz-Sobolev space setting, Bonanno, Molica Bisci and R\u{a}dulescu in \cite{Bonanno 1,Bonanno 2,Bonanno 3}  proved some important results concerning the existence of infinitely many weak solutions for a nonhomogeneous eigenvalue Dirichlet problem (see also \cite{Bonanno 4}).  Liu and  Zhao in \cite{P. Zhao}, obtained the existence of a nontrivial solution and infinitely many solutions  for a  quasilinear  equation related to problem $(P)$ in the framework of  Musielak-Sobolev spaces (see also \cite{X. Fan2}).

 In the above mentioned papers,   the authors assumed, among other conditions, that the nonlinearity $f$ satisfy to the well-known Ambrosetti-Rabinowitz  condition ((A-R)   condition for short); which  , for the $p$-Laplacian operator, asserts that
there exist two  constants $M>0$ and $\theta>p$, such that
$$   0<\theta F(x,t)\leq f(x,t)t, \quad \forall |t|\geq M,$$
where $F(x,t)=\int_{0}^{t}f(x,s)ds$. Clearly, this condition implies the existence of two positive constants $c_1, c_2$ such that
\begin{equation}\label{consequence of AR}
F(x,t)\geq c_1 |t|^{\theta} -c_2, \quad \forall (x,t)\in \Omega\times \mathbb{R}.
\end{equation}
This means that $f$ is $p$-superlinear at infinity in the sense that
\begin{equation}
\lim\limits_{|t|\to +\infty}\frac{F(x,t)}{|t|^{p}}= +\infty.
\end{equation}
This type of condition was introduced by Ambrosetti and Rabinowitz in their famous paper \cite{AR} and has since  become one of the main tools for finding solutions to elliptic problems of variational type; especially in order to prove the boundedness of Palais-Smale sequence of the energy functional associated with such a problem. Unfortunately, there
are several nonlinearities which are $p$-superlinear but do not satisfy the (A-R) condition.  For instance, if we take $f(x,t)=|t|^{p-2}t\ln(1+|t|)$, then we can check that for any $\theta>p$, $F(x,t)/|t|^{\theta}\to 0$ as $|t|\to +\infty$. However, many recent types of research have been made to drop the (A-R) condition (see e.g. \cite{Silva, H. Toan, Bin, V2} and references therein).

In \cite{Silva}, the authors studied a similar problem as that in \cite{Clément} and proved the existence of at least a nontrivial solution under the following assumptions on  the nonlinearity $f$:  there exist an $N$-function $\Gamma$ (cf. \cite{Rao}) and  positive constants $C,R$  such that
\begin{equation}\label{Silva}
\qquad \qquad \Gamma\left(\frac{F(x,t)}{|t|^{\phi_{0}}}\right)\leq C\bar{F}(x,t), \quad \forall (x,|t|)\in \Omega\times [R,+\infty),
\end{equation}
and
\begin{equation}\label{silva2}
\lim\limits_{|t|\to +\infty} \frac{f(x,t)}{|t|^{\phi^{0}-1}}=+\infty, \quad \lim\limits_{|t|\to 0} \frac{f(x,t)}{|t|\phi(t)}=\lambda,
\end{equation}
where  $\bar{F}(x,t):=f(x,t)t-\phi^{0}F(x,t)$, $\lambda$ some nonnegative constant and $\phi_{0}, \phi^{0}$ are defined  in relation (\ref{delta2}) below (when $\phi(x,t)=\phi(t)$ independent of $x$) with specific  assumptions.  It should be noted that the condition (\ref{Silva}) is a type of "nonquadraticity condition at infinity", which was first introduced by Costa and Magalh\~{a}es in  \cite{Costa} for the Laplacian operator   (with $\phi_{0}=\phi^{0}=2$) as follows:
 \begin{equation*}
 \liminf_{|t|\to +\infty} \frac{\bar{F}(x,t)}{|t|^{\sigma}}\geq a>0,
 \end{equation*}
  holds for some $\sigma>0$. We would  also like to mention that this condition  plays an important role in proving the boundedness of   Palais-Smale sequences.

   In \cite{H. Toan} also, the authors   considered a similar problem as that in \cite{Clément} and proved the existence of a  nontrivial solution under the following assumptions on the nonlinearity $f$: there exist  $\mu_{1}, \mu_{2}>0$ such that
  \begin{equation}\label{limits of f and F}
  \lim\limits_{|t|\to +\infty} \frac{F(x,t)}{|t|^{\phi^{0}}}=+\infty, \quad \lim\limits_{|t|\to 0} \frac{f(x,t)}{|t|^{\phi^{0}-1}}=0,
  \end{equation}
  \begin{equation}\label{monotonicity}
  \qquad \bar{F}(x,t)\leq \bar{F}(x,s) + \mu_1, \quad \forall (x,t)\in \Omega\times(0,s) \; \mbox{or } \; \forall(x,t)\in \Omega\times(s,0),
  \end{equation}
 and
  \begin{equation}
  \bar{H}(ts)\leq \bar{H}(t) + \mu_2, \quad \forall t\geq 0 \; \mbox{ and }\;  s\in [0,1],
  \end{equation}
    where $\bar{H}(t):=\phi^{0}\Phi(t)-\phi(t)t^{2}$ with $\Phi(t)=\int_{0}^{t}\phi(s)sds$.\\

  On the other hand, in the few last years, studies on double phase problems have attracted more and more interest and many results have been obtained. Especially,  in \cite{Bin} the authors proved the existence of a nontrivial solution and obtained infinitely many solutions  for a double phase problem without (A-R) condition. More precisely, they considered the problem $(P)$ (with $V\equiv 0$) with the function $\phi(x,t)=t^{p-2}+a(x)t^{q-2}$, where $a: \overline{\Omega}\mapsto [0,+\infty)$ is Lipschitz continuous, $1<p<q<N$, $\frac{q}{p}<1+\frac{1}{N}$ and  the nonlinearity $f$ satisfies the assumptions (\ref{limits of f and F}) and (\ref{monotonicity}) above with $\phi_{0}=p$ and $\phi^{0}=q$. In \cite{Bin2} however, the authors considered the same previous problem and proved the existence of infinitely many solutions; but instead of hypotheses (\ref{limits of f and F}) and (\ref{monotonicity}) the nonlinearity $f$ is supposed to satisfy the assumption (\ref{Silva}) above where $\Gamma(t)=|t|^{\sigma}$ with  $\sigma>\max\{1, \frac{N}{p}\}$,   and $F(x,t)\geq 0$ for any $(x,|t|)\in \Omega \times[R, +\infty)$ is such that $ \lim\limits_{|t|\to +\infty} \frac{F(x,t)}{|t|^{\phi^{0}}}=+\infty$.   In the same paper, the authors obtained also similar existence result under the following assumption instead of $(\ref{Silva})$:  there exist $\mu>q$ and $\theta>0$ such that
  \begin{equation*}
  \mu F(x, t) \leq tf (x, t) + \theta|t|^p, \quad\forall (x, t) \in \Omega\times \mathbb{R}.
  \end{equation*}

    Now, we give some  review results concerning the case where the potential $V\not\equiv 0$ on $\Omega$.
In \cite{Abdou}, Abdou and Marcos,  proved the existence of multiple solutions for a Dirichlet problem involving the $p(x)$-Laplacian operator with a changing sign  potential $V$ belonging to a generalized Lebesgue space $ L^{s(x)}(\Omega)$    when the nonlinearity $f$ satisfies some growth condition under (A-R) condition. In that work, the main assumptions on   the variable exponents $q(\cdot),s(\cdot)$ and $p(\cdot)$  are such that: $q,s,p\in \mathcal{C}_{+}(\overline{\Omega})$ (see notation below) and satisfy  $1<q(x)<p(x)\leq N<s(x)$ for any $x\in \overline{\Omega}$.

Recently, in \cite{V} the authors proved  the existence of nontrivial non-negative and non-positive solutions, and obtained infinitely many solutions for the quasilinear equation $- \mbox{div} A(x,\nabla u)+V(x)|u|^{\alpha(x)-2}u=f(x,u)$ in $\mathbb{R}^{N}$, where the divergence type operator has behaviors like $|\zeta|^{q(x)-2}$ for small $|\zeta|$ and like $|\zeta|^{p(x)-2}$ for large $|\zeta|$, where $1<\alpha(\cdot)\leq p(\cdot)<q(\cdot)<N$. In that paper,  it is supposed that the potential $V \in L_{loc}^{1}(\mathbb{R}^{N})$ verifies $V(\cdot)\geq V_{0}>0$, $V(x)\to +\infty$ as $|x|\to +\infty$ and that the nonlinearity $f$ satisfies some growth condition with the following assumption instead of (A-R) condition: there exist constants $M, C_1, C_2>0$ and  a function $a$ such that
\begin{equation}
C_1 |t|^{q(x)}[\ln (e+|t|)]^{a(x)-1}\leq C_2\frac{f(x,t)t}{\ln(e+|t|)}\leq f(x,t)t-s(x)F(x,t), \;\forall (x,|t|)\in \mathbb{R}^{N}\times [M,+\infty),
\end{equation}
where ess$_{x\in \mathbb{R}^{N}} \inf (a(x)-q(x))>0,$ $q(\cdot)\leq s(\cdot)$ and ess$_{x\in \mathbb{R}^{N}}\inf (p^{\ast}(x)-s(x))>0$ with $p^{\ast}(x):=\frac{Np(x)}{N-p(x)}$.
Related to this subject, we refer the readers to some important results concerning the study of the eigenvalue problems (see \cite{Benou,Bonanno 1,Bonanno 2,Bonanno 3,Bonanno 4,Kefi, Kim,Mihai2} and the references therein).

A main motivation of our current study is that, to the best of our knowledge, there is little research considering both the potential $V\not\equiv 0$ and nonlinearity $f$ without (A-R) condition for more general quasilinear equation in the framework  of Musielak-Sobolev spaces.  In this paper, our main goal is to show the existence  of weak solutions to the problem $(P)$. Firstly,  by using  standard lower semicontinuity argument, we prove the existence of weak solutions under the condition that $V\in L^{s(x)}(\Omega)$ has changing sign, and the nonlinearity $f$ satisfies the condition $(f_{0})$  below. Secondly, we establish the existence of at least a nontrivial solution and the existence of infinitely many solutions by using Mountain Pass Theorem and Fountain theorem respectively, where $V\in L^{s(x)}(\Omega)$ has constant sign and the nonlinearity $f$ does not satisfy the  (A-R) condition.  For these purposes, we propose a  set of growth conditions under which we are able to check the Palais-Smale condition. More precisely, we prove the boundedness of Palais-Smale sequences  by using a similar condition to that in (\ref{Silva}) above instead of (A-R) condition.

The paper is organized as follows: In Section 2, we recall some definitions and basic properties about Musielak-Orlicz-Sobolev spaces and variable exponent Lebesgue-Sobolev spaces. In Section 3, we state our main results and  in Section 4 we give the  proofs. Finally,   in Section 5, we give an application of our main results.
\section{Preliminary results}
In the study of nonlinear partial differential equations, it is well known that more general functional space can
handle differential equations with more nonlinearities. For example, the $p$-Laplacian  equations correspond to the classical Sobolev space setting, the $p(x)$-Laplacian equations correspond to the variable exponent Sobolev space setting, etc. Concerning the problem $(P)$, Musielak-Sobolev spaces are  the adequate functional spaces corresponding to the  solutions.  We shall therefore start by recalling some basic facts about  these spaces. For more details we refer  the readers to the papers \cite{Hudzik, Musielak,X. Fan, P. Zhao}.

Define
$$\Phi(x,t)=\int_{0}^{t}\phi(x,s)sds, \quad \forall t\geq0.$$

Since the function $\phi$ satisfies the condition $(\phi)$, then $\Phi$ is a generalized $N$-function, that is, for each $t\in [0, +\infty)$,  $\Phi(.,t)$ is measurable and for a.e. $x\in \Omega$, $\Phi(x,.)$ is continuous, even, convex, $\Phi(x,0)=0$, $\Phi(x,t)>0$ for $t>0$ which satisfies the conditions
$$\lim\limits_{t\to 0^{+}}\frac{\Phi(x,t)}{t}=0\quad\mbox{and}\lim\limits_{t\to +\infty}\frac{\Phi(x,t)}{t}=+\infty.$$

We denote by $N(\Omega)$ the set of generalized $N$-functions. For $\Phi \in N(\Omega),$ the Musielak-Orlicz space $L^{\Phi}(\Omega)$ is defined by
$$L^{\Phi}(\Omega):=\left\{ u : u: \Omega \to \mathbb{R} \mbox{ is measurable, and } \exists \lambda>0 \mbox{ such that} \int_{\Omega}\Phi\left(x, \frac{|u(x)|}{\lambda}\right)dx < \infty \right\}$$
endowed with the Luxemburg norm
$$\|u\|_{L^{\Phi}(\Omega)}=\|u\|_{\Phi}:=\inf\left\{ \lambda>0 : \int_{\Omega}\Phi\left(x, \frac{|u(x)|}{\lambda}\right)dx \leq 1\right\}.$$
Next, we define  the Musielak-Sobolev space $W^{1,\Phi}(\Omega)$ by
$$W^{1,\Phi}(\Omega):=\left\{u\in L^{\Phi}(\Omega) : |\nabla u|\in L^{\Phi}(\Omega)\right\}$$
endowed with the norm
$$\|u\|_{W^{1,\Phi}(\Omega)}=\|u\|_{1,\Phi}:=\|u\|_{\Phi}+\|\nabla u\|_{\Phi},$$
where $\|\nabla u\|_{\Phi}=\||\nabla u|\|_{\Phi}$.
\begin{remark}
	In the particular case where $\Phi(x,t)=\Phi(t)$ is independent of $x$,   $W^{1,\Phi}(\Omega)$ is actually an Orlicz-Sobolev space while in the case where $\Phi(x,t)=|t|^{p(x)}$, this space   becomes the variable exponent Sobolev space $W^{1,p(.)}(\Omega)$.
\end{remark}
 The function $\tilde{\Phi} : \Omega\times [0, +\infty)\to [0, +\infty)$ defined by
$$\tilde{\Phi}(x,t)=\sup_{s> 0}(ts-\Phi(x,s)), \quad \mbox{ for } x\in \Omega \mbox{ and } t\geq0,$$ is called the complementary function to $\Phi$ in the sense of  Young. We observe that the function $\tilde{\Phi}$   belongs to $N(\Omega)$, and $\Phi$ is also  the complementary function to $\tilde{\Phi}$. Furthermore, $\Phi$ and $\tilde{\Phi}$ satisfy the Young inequality
$$st\leq \Phi(x,t) + \tilde{\Phi}(x,s), \quad \mbox{ for } x\in \Omega \mbox{ and } s,t \geq 0.$$

Throughout this paper, we assume  that there exist two positive constants $\phi_{0}$ and $\phi^{0}$ such that
\begin{equation}\label{delta2}
1<\phi_{0}\leq \frac{\phi(x,t)t^{2}}{\Phi(x,t)}\leq \phi^{0}<N, \quad \mbox{ for } x\in \Omega \mbox{ and } t>0.
\end{equation}
This relation gives the following result (see \cite[Proposition 2.1]{Mihai}):
\begin{lemma}\label{min-max for Phi}
	Let $u\in L^{\Phi}(\Omega)$ and $\rho,t\geq 0,$ then we have
	\begin{equation}
	\min\left\{\rho^{\phi_0},\rho^{\phi^{0}}\right\}\Phi(x,t)\leq\Phi(x,\rho t) \leq \max\left\{\rho^{\phi_0},\rho^{\phi^{0}}\right\}\Phi(x,t),
	\end{equation}
	\begin{equation}
	\min\left\{\|u\|^{\phi_0}_{\Phi},\|u\|^{\phi^{0}}_{\Phi}\right\}\leq\int_{\Omega}\Phi(x,|u(x)|)dx \leq \max\left\{\|u\|^{\phi_0}_{\Phi},\|u\|^{\phi^{0}}_{\Phi}\right\}.
	\end{equation}
\end{lemma}
Using the previous lemma we can  easily show the following result:
\begin{proposition}
	The function  $\Phi$  satisfies the $(\Delta_{2})$-condition, that is,  there exist a positive constant $C>0$ such that
	$$\Phi(x,2t)\leq C\Phi(x,t), \quad \mbox{ for } x\in \Omega \mbox{ and } t\geq0.$$
\end{proposition}
For the complementary function $\tilde{\Phi}$ we have the following analog lemma  (see \cite{Fukagai}).
\begin{lemma}\label{min-max for Phi tilde}
	Let $u\in L^{\tilde{\Phi}}(\Omega)$ and $\rho,t\geq 0,$ then we have
	\begin{equation}
	\min\left\{\rho^{({\phi_0})^{\prime}},\rho^{({\phi^0})^{\prime}}\right\}\tilde{\Phi}(x,t)\leq\tilde{\Phi}(x,\rho t) \leq \max\left\{\rho^{({\phi_0})^{\prime}},\rho^{({\phi^0})^{\prime}}\right\}\tilde{\Phi}(x,t),
	\end{equation}
	\begin{equation}
	\min\left\{\|u\|^{({\phi_0})^{\prime}}_{\tilde{\Phi}},\|u\|^{({\phi^0})^{\prime}}_{\tilde{\Phi}}\right\}\leq\int_{\Omega}\tilde{\Phi}(x,|u(x)|)dx \leq \max\left\{\|u\|^{({\phi_0})^{\prime}}_{\tilde{\Phi}},\|u\|^{({\phi^0})^{\prime}}_{\tilde{\Phi}}\right\},
	\end{equation}
	where $({\phi_0})^{\prime}=\frac{\phi_0}{\phi_0-1}$ and $({\phi^0})^{\prime}=\frac{\phi^0}{\phi^0-1}$.
\end{lemma}
\begin{remark}
	From Lemma \ref{min-max for Phi tilde},    the complementary function $\tilde{\Phi}$ also satisfies $(\Delta_{2})$-condition.
\end{remark}
 Since both $\Phi$ and $\tilde{\Phi}$ satisfy the $(\Delta_{2})$-condition, then we have the following result:
\begin{proposition}[\cite{P. Zhao}] \label{proposition of Phi}
	The following assertions hold:
\begin{enumerate}
	\item $L^{\Phi}(\Omega)=\{u : u: \Omega \to \mathbb{R} \mbox{ is measurable, and }  \int_{\Omega}\Phi\left(x, |u(x)|\right)dx < \infty\}$
	\item For any sequence $(u_n)$ in $L^{\Phi}(\Omega)$, we have
	\begin{enumerate}
		\item[a)] $\int_{\Omega}\Phi\left(x, |u_{n}(x)|\right)dx \to 0 (\mbox{resp. } 1; +\infty)\Leftrightarrow \|u_n\|_{\Phi}\to 0 (\mbox{resp. } 1; +\infty),$
		\item[b)] $u_{n}\to u$ in $L^{\Phi}(\Omega)\Rightarrow \int_{\Omega}|\Phi\left(x, |u_{n}(x)|\right)-\Phi\left(x, |u(x)|\right)|dx\to 0 $ as $n\to +\infty$.
	\end{enumerate}
	\item Let $u\in L^{\Phi}(\Omega) $ and $v\in L^{\tilde{\Phi}}(\Omega)$. Then the H\"{o}lder type inequality holds true
	$$ \left|\int_{\Omega} u(x)v(x)dx\right|\leq 2\|u\|_{\Phi} \|v\|_{\tilde{\Phi}}.$$
	\item $\phi(x,|u(x)|)u(x)\in L^{\tilde{\Phi}}(\Omega)$ provided that $u\in L^{\Phi}(\Omega)$.
\end{enumerate}	
\end{proposition}
Let $\Phi, \Psi\in N(\Omega)$. We say that  $\Phi$ is weaker than $\Psi$, and denote $\Phi\preccurlyeq \Psi$, if there exist  positive constants $K_1, K_2$ and a nonnegative function $h\in L^{1}(\Omega)$ such that
$$\Phi(x,t)\leq K_{1}\Psi(x,K_{2}t)+h(x), \quad \mbox{for } x\in \Omega \mbox{ and } t\geq 0.$$
By Theorem 8.5 in \cite{Musielak},   the following  embeddings
$$L^{\Psi}(\Omega)\hookrightarrow L^{\Phi}(\Omega) \mbox{ and } L^{\tilde{\Phi}}(\Omega) \hookrightarrow L^{\tilde{\Psi}}(\Omega)$$ are continuous provided that  $\Phi, \Psi$ are  such that  $\Phi\preccurlyeq \Psi$.

  We say that $\Phi \in N(\Omega)$ is  locally integrable if $\Phi(.,t_{0})\in L^{1}(\Omega)$ for every $t_{0}>0$. We point out that when  $\Phi$ is locally integrable, then $(L^{\Phi}(\Omega),\|\cdot\|_{\Phi} )$ is a separable Banach space (see \cite{Hudzik, Musielak}).

In this paper, we shall need the following assumptions.
\begin{enumerate}
	\item[$(\phi_{1})$] $\inf_{x\in \Omega}\Phi(x,1)=c_{1}>0.$
     \item[$(\phi_{2})$] For every $t_{0}>0$ there exists $c=c(t_0)>0$ such that
     $$\frac{\Phi(x,t)}{t}\geq c, \quad \mbox{and}\quad \frac{\tilde{\Phi}(x,t)}{t}\geq c\quad \mbox{for } x\in \Omega \mbox{ and } t\geq t_0. $$
\end{enumerate}
We note that, $(\phi_{2})\Rightarrow (\phi_{1})$. Moreover,  in the case where $\Phi$ is independent of $x$,  $(\phi_{1})$ and $(\phi_{2})$ hold automatically and $\Phi$ is automatically locally integrable.

By assumption $(\phi_{1})$, we have the following embeddings
 $$L^{\Phi}(\Omega) \hookrightarrow L^{1}(\Omega)\quad \mbox{and}\quad W^{1,\Phi}(\Omega) \hookrightarrow W^{1,1}(\Omega) .$$

In addition, by  assuming   $\Phi$ and $\tilde{\Phi}$    both locally integrable and  satisfy $(\phi_{2})$, we conclude that   $L^{\Phi}(\Omega)$ is reflexive, and that the mapping $J : L^{\tilde{\Phi}}(\Omega)\to (L^{\Phi}(\Omega))^{\ast} $ defined by
$$\langle J(v),w \rangle=\int_{\Omega}v(x)w(x)dx, \quad \forall v\in L^{\tilde{\Phi}}(\Omega), \forall w\in L^{\Phi}(\Omega) $$
is a linear isomorphism and $\|J(v)\|_{(L^{\Phi}(\Omega))^{\ast} }\leq 2\|v\|_{L^{\tilde{\Phi}}(\Omega)}$ (see \cite[p. 189]{Musielak}.)

We  denote by $W_{0}^{1,\Phi}(\Omega)$ the closure of $\mathcal{C}_{0}^{\infty}(\Omega)$ in $W^{1,\Phi}(\Omega)$ and by $\mathcal{D}^{1,\Phi}_{0}(\Omega)$ the completion of $\mathcal{C}_{0}^{\infty}(\Omega)$ in the norm $\|\nabla u\|_{\Phi}$. It is clear that  $W_{0}^{1,\Phi}(\Omega)=\mathcal{D}^{1,\Phi}_{0}(\Omega)$ in the case where $\|\nabla u\|_{\Phi}$ is an equivalent norm in $W_{0}^{1,\Phi}(\Omega)$ .

By assuming   $\Phi$  locally integrable and satisfies $(\phi_{1})$,    $W^{1,\Phi}(\Omega), W_{0}^{1,\Phi}(\Omega)$ and $\mathcal{D}^{1,\Phi}_{0}(\Omega)$ are clearly separable Banach spaces, and we have
\begin{align*}
W_{0}^{1,\Phi}(\Omega)&\hookrightarrow W^{1,\Phi}(\Omega)\hookrightarrow W^{1,1}(\Omega)\\
\mathcal{D}^{1,\Phi}_{0}(\Omega)&\hookrightarrow \mathcal{D}^{1,1}_{0}(\Omega)=W_{0}^{1,1}(\Omega).
\end{align*}
In addition, these spaces are reflexive if $L^{\Phi}(\Omega)$ is reflexive.

In this work, we need to use some standard   tools such as the Poincaré inequality and   results of compactness for embeddings in  Musielak-Sobolev spaces  . For this reason, we shall suppose the following   supplementary  assumptions on $\Phi$.
\begin{enumerate}
	\item[($H_{1})$] $\Omega\subset \mathbb{R}^{N} (N\geq 2)$ is a bounded domain with the cone property, $\Phi\in N(\Omega)$.
     \item[$(H_2)$]  $\Phi : \overline{\Omega}\times [0,+\infty)\to [0,+\infty)$ is continuous and $\Phi(x,t)\in (0,+\infty)$ for $x\in\overline{\Omega}$ and $t\in(0,+\infty)$.
    \end{enumerate}
     Now, let $\Phi$ satisfies $(H_1)$ and $(H_2)$. Then, for each $x\in \overline{\Omega}$, the function $\Phi(x,\cdot):[0,+\infty)\to [0,+\infty) $  is a  strictly increasing homeomorphism.  Denote by $\Phi^{-1}(x,\cdot)$ the inverse function of $\Phi(x, \cdot)$. We also assume  the following condition.
     \begin{enumerate}
     	\item[$(H_3)$]
     	$$ \int_{0}^{1}\frac{\Phi^{-1}(x,t)}{t^{\frac{N+1}{N}}}dt<+\infty, \;\forall x\in \overline{\Omega}. $$
     \end{enumerate}
  Define the function $\Phi^{-1}_{\ast} : \overline{\Omega}\times [0,+\infty)\to [0,+\infty) $ by
  \begin{equation}
  \Phi^{-1}_{\ast}(x,s)=\int_{0}^{s}\frac{\Phi^{-1}(x,\tau)}{\tau^{\frac{N+1}{N}}}d\tau, \; \mbox{ for } x\in \overline{\Omega} \;\mbox{ and }\; s\in [0,+\infty).
  \end{equation}
  Then, by   assumption $(H_3)$, $\Phi^{-1}_{\ast}$  is well defined, and for each $x\in \overline{\Omega}$,  $\Phi^{-1}_{\ast}(x,\cdot)$  is strictly increasing, $\Phi^{-1}_{\ast}(x,\cdot)\in \mathcal{C}^{1}((0,+\infty))$ and the function $\Phi^{-1}_{\ast}(x,\cdot)$ is concave.

  Set $T(x)=\lim\limits_{s\to +\infty}\Phi^{-1}_{\ast}(x,s),$ for all $ x\in \overline{\Omega}$. Then, $T(x)\in (0,+\infty]$. Define the function $\Phi_{\ast} : \overline{\Omega}\times [0,+\infty)\to [0,+\infty)$ by
  $$\Phi_{\ast}(x,t)=\begin{cases}
  s,& \mbox{ if }\; x\in \overline{\Omega},\; t\in [0,T(x))\; \mbox{ and }\; \Phi^{-1}_{\ast}(x,s)=t\\
  +\infty, & \mbox{ if }\; x\in \overline{\Omega},\; t\geq T(x).
  \end{cases}
  $$
  Then $\Phi_{\ast} \in N(\Omega),$ and for each $x\in \overline{\Omega}$, $\Phi_{\ast}(x,\cdot)\in \mathcal{C}^{1}((0,T(x)))$. $\Phi_{\ast}$ is called the Sobolev conjugate function of $\Phi$.

  Let $X$ be a metric space and $f : X\to (-\infty, +\infty]$ be an extended real-valued function. For $x\in X$  with $f(x)\in\mathbb{R}$, the continuity of $f$ at $x$ is well defined. Now, for $x\in X$ with $f(x)=+\infty$, we say that $f$ is continuous at $x$ if given any $M>0$, there exists a neighborhood $U$ of $x$ such that $f(y)>M$ for all $y\in U$. We say that $f : X\to (-\infty, +\infty]$ is continuous on $X$ if $f$ is continuous at every $x\in X$. Define $Dom(f)=\{x\in X : f(x)\in \mathbb{R}\}$ and denote by $\mathcal{C}^{1-0}(X)$ the set of all locally Lipschitz continuous real-valued functions defined on $X$.
  \begin{remark}\label{bounded of K}
  	Suppose that $\Phi$ satisfies $(H_2)$.  Then, for each  $t_0\geq 0$, $\tilde{\Phi}(x,t_0)$ and  $\Phi_{\ast}(x,t_0)$ are bounded.
  \end{remark}

  Concerning the function $\Phi_{\ast}$ and the operator $T$ , we suppose that  
  \begin{enumerate}
  	\item[$(H_4)$] $T : \overline{\Omega} \to [0,+\infty]$ is continuous on $\overline{\Omega}$ and $T\in \mathcal{C}^{1-0}(Dom(T))$;
  	\item[$(H_5)$] $\Phi_{\ast} \in \mathcal{C}^{1-0}(Dom(\Phi_{\ast}))$ and there exist positive constants  $C_0, \delta_0<\frac{1}{N}$ and $t_0\in (0,\min_{x\in \overline{\Omega}}T(x))$ such that
  	$$|\nabla_{x}\Phi_{\ast}(x,t)|\leq C_{0}(\Phi_{\ast}(x,t))^{1+\delta_0},$$
  	for $x\in \Omega$ and $t\in [t_0, T(x))$ provided that $\nabla_{x}\Phi_{\ast}(x,t)$ exists.\\
  \end{enumerate}

  \begin{remark}
  Examples of generalized $N$-function $\Phi$ satifying the above assumptions and  covering the case of   variable exponent space,  double-phase space, and variable exponent double-phase space, are given in \cite{P. Zhao 2}.
  \end{remark}

  Let $\Phi, \Psi\in N(\Omega)$. We say that  $\Phi$ essentially grows more slowly that  $\Psi$ and we write $\Phi\ll \Psi$, if  for any $k>0,$
  $$\lim\limits_{t\to +\infty}\frac{\Phi(x,kt)}{\Psi(x,t)}=0 \mbox{ uniformly for } x\in \Omega.$$
  Obviously, if $\Phi \ll \Psi$ then $ \Phi \preccurlyeq \Psi$.

  Now, we recall the following embedding theorems for Musielak-Sobolev spaces (see \cite{X. Fan,P. Zhao}).
  \begin{theorem}\label{compact embedding theorem1 }
  	Assume $(H_1)$-$(H_5)$ hold. Then
  	\begin{enumerate}
  	\item There is a continuous embedding $W^{1,\Phi}(\Omega)\hookrightarrow L^{\Phi_{\ast}}(\Omega)$.
  		\item Suppose that $\Psi\in N(\Omega)$, $\Psi : \overline{\Omega}\times [0,+\infty)\to [0,+\infty) $ is continuous, and $\Psi(x,t)\in (0,+\infty)$ for $x\in \Omega$ and $t\in(0,+\infty)$. If $\Psi\ll \Phi_{\ast}$. Then,  the embedding $W^{1,\Phi}(\Omega)\hookrightarrow \hookrightarrow L^{\Psi}(\Omega)$ is compact.
  	\end{enumerate}
  \end{theorem}
In particular, as $\Phi\ll \Phi_{\ast}$ then we have
\begin{theorem}\label{compact embedding theorem2}
		Assume $(H_1)$-$(H_5)$ hold. Then
	\begin{enumerate}
	  \item  The embedding $W^{1,\Phi}(\Omega)\hookrightarrow \hookrightarrow L^{\Phi}(\Omega)$ is compact .
	  \item The Poincaré type inequality
	  $$\|u\|_{\Phi}\leq C\|\nabla u\|_{\Phi} \quad \mbox{ for } u\in W_{0}^{1,\Phi}(\Omega),$$  holds.
	\end{enumerate}
\end{theorem}
We finish the recall of Musielak-Sobolev spaces properties by giving the following  analog lemma (see  \cite{Fukagai}).
\begin{lemma}\label{min-max for Phi ast}
	Let $u\in L^{\Phi_{\ast}}(\Omega)$ and $\rho,t\geq 0.$ Then, we have
	\begin{equation}
	\min\left\{\rho^{({\phi_0})^{\ast}},\rho^{({\phi^0})^{\ast}}\right\}\Phi_{\ast}(x,t)\leq\Phi_{\ast}(x,\rho t) \leq \max\left\{\rho^{({\phi_0})^{\ast}},\rho^{({\phi^0})^{\ast}}\right\}\Phi_{\ast}(x,t),
	\end{equation}
	\begin{equation}
	\min\left\{\|u\|^{({\phi_0})^{\ast}}_{\Phi_{\ast}},\|u\|^{({\phi^0})^{\ast}}_{\Phi_{\ast}}\right\}\leq\int_{\Omega}\Phi_{\ast}(x,|u(x)|)dx \leq \max\left\{\|u\|^{({\phi_0})^{\ast}}_{\Phi_{\ast}},\|u\|^{({\phi^0})^{\ast}}_{\Phi_{\ast}}\right\},
	\end{equation}
	where $({\phi_0})^{\ast}=\frac{N\phi_0}{N-\phi_0}$ and $({\phi^0})^{\ast}=\frac{N\phi^0}{N-\phi^0}$.\\
\end{lemma}

Now, we give some background facts concerning the variable exponent Lebesgue spaces. For more details on the basic properties of these spaces, we refer the reader to the papers \cite{D. Zhao, Kovacik}.

Set $$\mathcal{C}_{+}(\overline{\Omega})=\{h\in \mathcal{C}(\overline{\Omega}) : h(x)>1 \mbox{ for any } x\in \overline{\Omega}\}.$$
For any $h\in \mathcal{C}_{+}(\overline{\Omega})$ we define:
$$h^{-}=\min_{x\in \overline{\Omega}}h(x), \quad h^{+}=\max_{x\in \overline{\Omega}}h(x).$$
For any $q(x)\in \mathcal{C}_{+}(\overline{\Omega})$, we define the variable exponent Lebesgue space:
$$L^{q(x)}(\Omega):=\left\{ u : u: \Omega \to \mathbb{R} \mbox{ is measurable with } \int_{\Omega}|u(x)|^{q(x)}dx<\infty\right\},$$
endowed with the norm
$$\|u\|_{L^{q(x)}(\Omega)}= \|u\|_{q(x)} :=\inf\left\{\lambda>0: \; \int_{\Omega}\left|\frac{u(x)}{\lambda}\right|^{q(x)}dx\leq 1 \right\}.$$

 We recall that the variable exponent Lebesgue spaces are separable and reflexive Banach spaces. Let $L^{q^{\prime}(x)}(\Omega)$ denote the conjugate space of $L^{q(x)}(\Omega)$  with $\frac{1}{q(x)}+\frac{1}{q^{\prime}(x)}=1$. For any $u\in L^{q(x)}(\Omega)$ and $v\in L^{q^{\prime}(x)}(\Omega)$, the H\"{o}lder type inequality
 \begin{equation}
 \int_{\Omega}|uv|dx\leq \left(\frac{1}{q^{-}}+\frac{1}{q^{\prime -}}\right)\|u\|_{q(x)}\|v\|_{q^{\prime}(x)},
 \end{equation}
  holds true. Moreover, if $0<|\Omega|<\infty$ and $q_1, q_2$   are variable exponents so that $q_{1}(x)\leq q_{2}(x)$  almost everywhere in $\Omega$, then there exists the continuous embedding  $L^{q_{2}(x)}(\Omega)\hookrightarrow L^{q_{1}(x)}(\Omega)$. Furthermore, if we define the mapping $\rho_{q} : L^{q(x)}(\Omega) \to \mathbb{R}^{+}$ by
  $$\rho_{q}(u)=\int_{\Omega}|u|^{q(x)}dx,$$
  then the following relations hold true:
   \begin{align}\label{min and max}
   \min\{\|u\|^{q_{-}}_{q(x)},\|u\|^{q_{+}}_{q(x)}\}\leq \rho_{q}(u) \leq \max\{\|u\|^{q_{-}}_{q(x)},\|u\|^{q_{+}}_{q(x)}\},
   \end{align}
  \begin{align}
     \|u\|_{q(x)}<1(=1, >1) \Leftrightarrow \rho_{q}(u)<1(=1,>1),
  \end{align}
  \begin{align}
  \|u_{n}-u\|_{q(x)} \to 0 \Leftrightarrow \rho_{q}(u_{n}-u)\to 0,\quad \forall (u_n), u\in L^{q(x)}(\Omega),
  \end{align}
  \begin{align}
  \|u_{n}\|_{q(x)} \to +\infty \Leftrightarrow \rho_{q}(u_{n})\to  +\infty ,\quad \forall (u_n)\in L^{q(x)}(\Omega).
  \end{align}

We recall also the following proposition, which will be used later:
\begin{proposition}[\cite{Edmunds}]\label{proposition inequality}
	Let $p$ and $q$ be measurable functions such that $p\in L^{\infty}(\Omega)$ and $1<p(x)q(x)\leq \infty$ for a.e. $x\in \Omega$. Let $u\in L^{q(x)}(\Omega), u\neq 0$. Then
	\begin{align*}
	\|u\|_{p(x)q(x)}\leq 1 \Rightarrow \|u\|^{p^{+}}_{p(x)q(x)}\leq \||u|^{p(x)}\|_{q(x)}\leq \|u\|^{p^{-}}_{p(x)q(x)},\\
	\|u\|_{p(x)q(x)}\geq 1 \Rightarrow \|u\|^{p^{-}}_{p(x)q(x)}\leq \||u|^{p(x)}\|_{q(x)}\leq \|u\|^{p^{+}}_{p(x)q(x)}.
	\end{align*}
	In particular when $p(x)=p$ is a constant, then
	$$\||u|^{p}\|_{q(x)}=\|u\|^{p}_{pq(x)}.$$
\end{proposition}
\section{Main results}
In this section we state the main results of this paper.  We will study the problem $(P)$ when $q\in \mathcal{C}_{+}(\overline{\Omega})$ and the potential $V: \Omega \to \mathbb{R}$ is nontrivial and belongs to $L^{s(x)}(\Omega)$ with $s\in \mathcal{C}(\overline{\Omega})$. Before dealing with our main results in this section, we introduce the following assumptions  for $f(x,u)$:
\begin{enumerate}
	\item[$(f_0)$] There exists $\Psi\in N(\Omega)$  satisfying the assumption (2) of Theorem \ref{compact embedding theorem1 }, and two positive constants $\psi_{0}$ and $\psi^{0}$  such that
	\begin{align}\label{assumption f0}
	1<\psi_{0}\leq \frac{\psi(x,t)t}{\Psi(x,t)}\leq \psi^{0},  \mbox{ for } x\in \Omega \mbox{ and } t>0.
	\end{align}
	\begin{align}\label{growth of f}
	|f(x,t)|\leq C_1\psi(x,|t|)+h(x), \mbox{ for } (x,t)\in \Omega\times \mathbb{R},
	\end{align}
	where $C_1$ is a positive constant, $0\leq h\in L^{\tilde{\Psi}}(\Omega)$, and $ \psi : \Omega \times \mathbb{R}^{+} \to \mathbb{R}^{+}$  is a continuous function  and $ \Psi(x,t)=\int_{0}^{t}\psi(x,s)ds,$ for all $x\in \Omega$.
	\item[$(f_1)$] There exists $\Gamma\in N(\Omega)$ satisfying the assumptions of ($H_2$), and two positive constants $\gamma_{0}$ and $\gamma^{0}$  such that
	\begin{align}\label{gamma condition}
	1<\frac{N}{\phi_{0}}<\gamma_{0}\leq \frac{\gamma(x,t)t}{\Gamma(x,t)}\leq \gamma^{0}, \mbox{ for } x\in \Omega \mbox{ and } t>0.\\
	\Gamma\left(x,\frac{F(x,t)}{|t|^{\phi_{0}}}\right)\leq C_2 H(x,t), \mbox{ for } x\in \Omega \mbox{ and } |t|\geq M,\label{Malghass condition}
	\end{align}
	where $C_2,M$ are positive constants, $H(x,t)=f(x,t)t-\nu F(x,t),$ for all  $(x,t)\in\Omega\times \mathbb{R}$ with $\nu=\phi^{0}$ if $V\leq 0$ a.e. on $\Omega$ and $\nu=q^{+}$ if $V\geq 0$ a.e. on $\Omega$, and $ \gamma : \Omega \times \mathbb{R}^{+} \to \mathbb{R}^{+}$  is a continuous function  and $ \Gamma(x,t)=\int_{0}^{t}\gamma(x,s)ds,$ for all $x\in \Omega$.
	\item[$(f_2)$] $\lim\limits_{|t|\to +\infty}\frac{F(x,t)}{|t|^{\phi^{0}}}= +\infty,$ uniformly for $x \in \Omega$.
	\item[$(f_3)$] $f(x,t)=o(|t|\phi(x,t))$ as $ t\to 0,$ uniformly for $x\in \Omega$.
	\item[$(f_4)$] $f(x,-t)=-f(x,t)$ for all $(x,t)\in \Omega\times \mathbb{R}.$
\end{enumerate}

\vspace*{0.4cm}
To summarize all assumptions concerning the function $\Phi$, in what follows  we shall say    that the function $\Phi$ satisfies the assumption $(\Phi)$ if:
 $\phi$ satisfies the assumption $(\phi)$,  $\Phi$ satisfies $(\ref{delta2})$  and $(H_{1})$-$(H_{5})$,  both $\Phi$ and $\tilde{\Phi}$ are locally integrable and satisfy $(\phi_{2})$. Hence, under the assumption $(\Phi)$, the spaces $L^{\Phi}(\Omega)$, $W^{1,\Phi}(\Omega)$, $W_{0}^{1,\Phi}(\Omega)$ are separable reflexive Banach spaces. Therefore, we can apply  the  embedding theorems for Musielak-Sobolev spaces in Theorem \ref{compact embedding theorem1 } and Theorem \ref{compact embedding theorem2}.
\begin{definition}\label{definition of weak solution}
	 A function $u\in W_{0}^{1,\Phi}(\Omega)$  is said to be a weak solution of problem $(P)$ if it holds that
	 $$\int_{\Omega}\phi(x,|\nabla u|)\nabla u \nabla v dx+ \int_{\Omega}V(x)|u|^{q(x)-2}uvdx=\int_{\Omega}f(x,u)vdx,\quad \forall v\in W_{0}^{1,\Phi}(\Omega).$$
\end{definition}
Our main results in this paper are given by the following  theorems:
\begin{theorem}\label{main theorem 1}
 Assume that the assumptions $(\Phi)$ and $(f_0)$ hold. Furthermore,   assume that $\max\{\psi^{0}, q^{+} \}<\phi_{0}$ and $s(x)>\frac{q(x)(\phi_{0})^{\ast}}{(\phi_{0})^{\ast}-q(x)}$ for every $x\in\overline{\Omega}$. Then,  problem $(P)$ has a weak solution.	
\end{theorem}
In order to obtain the second main result, we  assume that $f$ satisfies  the following condition $(f^{\prime}_{0})$ instead of $(f_{0})$:
\vspace*{0.5mm}
\\
\begin{enumerate}
	\item[$(f^{\prime}_{0})$] We assume that (\ref{assumption f0}) of $(f_0)$ holds and that
	\begin{align}
	|f(x,t)|\leq C_1(\psi(x,|t|)+1), \mbox{ for } (x,t)\in \Omega\times \mathbb{R},
	\end{align}
	where $C_1$ is a positive constant.	
\end{enumerate}

\begin{theorem}\label{main theorem 2}
	Assume that the assumptions $(\Phi)$ and $(f^{\prime}_0)$-$(f_3)$ hold. Furthermore,  assume that $\phi^{0}<\min\{\psi_{0},q^{-}\},$ $\max\{\psi^{0},q^{+}\}<(\phi_{0})^{\ast},$ $q^{+}-\frac{1}{2}\phi_{0}<q^{-}$ and $s(x)>\frac{q(x)(\phi_{0})^{\ast}}{(\phi_{0})^{\ast}-q(x)}$ for every $x\in \overline{\Omega}$. If $V$ has a constant sign a.e. on $\Omega$, 	then the problem $(P)$ has a nontrivial weak solution.
\end{theorem}
\begin{theorem}\label{main theorem 3, Fountain}
	 	Assume that the assumptions of Theorem \ref{main theorem 2} hold. If the function $f$ satisfies $(f_4)$, then the problem $(P)$ has a sequence of weak solutions $(\pm u_{n})_{n\in \mathbb{N}} \subseteq W_{0}^{1,\Phi}(\Omega) $ such that $\mathcal{I}(\pm u_n)\to +\infty$ as $n\to+\infty$.
\end{theorem}
In order to prove Theorem \ref{main theorem 3, Fountain} we will use the following Fountain theorem (see \cite{Willem} for details). Let $(X,\|\cdot\|)$ be a real  reflexive Banach space such that  $X=\overline{\oplus_{j\in \mathbb{N^{\ast}}} X_{j}}$ with $dim(X_j)<+\infty$ for any $j \in \mathbb{N^{\ast}}.$ For each $k \in \mathbb{N^{\ast}}$, we set $Y_k=\oplus^{k}_{j=1}X_{j}$ and $Z_{k}=\overline{\oplus^{\infty}_{j=k}X_{j}}$.
\begin{proposition}[Fountain theorem]\label{Fountain}
	Let $(X,\|\cdot\|)$ be a real reflexive Banach space and $\mathcal{I}\in \mathcal{C}^{1}(X,\mathbb{R})$  an even functional. If for each sufficiently large $k \in  \mathbb{N^{\ast}},$ there exist $\rho_{k}>r_k >0$ such that the following conditions hold:
	\begin{enumerate}
		\item $\inf_{\{u\in Z_{k}, \|u\|=r_k \}}\mathcal{I}(u) \to +\infty$ as $k\to +\infty$,
		\item $\max_{\{u\in Y_{k}, \|u\|=\rho_k \}}\mathcal{I}(u) \leq 0$,
		\item $\mathcal{I}$ satisfies the Palais-Smale condition for every $c>0$,
	\end{enumerate}
then $\mathcal{I}$ has a sequence of critical values tending to $+\infty$.
\end{proposition}
\section{Proofs of the main results}
In this section we give the proofs of our main results. We note that in these results we always have $s(x)>\frac{q(x)(\phi_{0})^{\ast}}{(\phi_{0})^{\ast}-q(x)}$ for every $x\in\overline{\Omega}$ and $\max\{\psi^{0}, q^{+} \}<(\phi_{0})^{\ast}$.

 Define the functional $\mathcal{I} :W_{0}^{1,\Phi}(\Omega)\to \mathbb{R} $ by the formula
\begin{equation}\label{definition of I}
\mathcal{I}(u)= \mathcal{H}(u)+\mathcal{J}(u)-\mathcal{F}(u),
\end{equation}
where,
$$\mathcal{H}(u)=\int_{\Omega}\Phi(x,|\nabla u|)dx,\quad \mathcal{J}(u)=\int_{\Omega}\frac{V(x)}{q(x)}|u|^{q(x)}dx, \quad\mbox{and}\quad \mathcal{F}(u)=\int_{\Omega}F(x,u)dx,$$
with $F(x,t)=\int_{0}^{t}f(x,s)ds.$
\begin{proposition}\label{the functional I}
	The functional $\mathcal{I}$ is well defined and $\mathcal{I}\in \mathcal{C}^{1}(W_{0}^{1,\Phi}(\Omega),\mathbb{R})$ with the derivative given by
	\begin{equation*}\label{derivative of I}
	\langle\mathcal{I}^{\prime}(u),v \rangle= \int_{\Omega}\phi(x,|\nabla u|)\nabla u \nabla v dx+ \int_{\Omega}V(x)|u|^{q(x)-2}uvdx-\int_{\Omega}f(x,u)vdx,\quad \forall u, v\in W_{0}^{1,\Phi}(\Omega).
	\end{equation*}
\end{proposition}
\begin{proof}
	Firstly, it is clear that $\mathcal{H}$ is well defined on $W_{0}^{1,\Phi}(\Omega)$. Furthermore, by similar  arguments  used in the proof of  \cite[Lemma 4.2]{Mihai}, we have   $\mathcal{H}\in \mathcal{C}^{1}(W_{0}^{1,\Phi}(\Omega),\mathbb{R})$ and its derivative is given by $$\langle\mathcal{H}^{\prime}(u),v \rangle= \int_{\Omega}\phi(x,|\nabla u|)\nabla u \nabla vdx, \quad\forall u, v\in W_{0}^{1,\Phi}(\Omega).$$
	Secondly, the functional $\mathcal{J}$ is well defined. Indeed, since $s(x)>\frac{q(x)(\phi_{0})^{\ast}}{(\phi_{0})^{\ast}-q(x)}$ for every $x\in\overline{\Omega}$, then it is clear that $s\in \mathcal{C}_{+}(\overline{\Omega})$ and $s(x)>q(x)$ for every $x\in\overline{\Omega}$. Furthermore, by a simple computation we have,
	\begin{align}
		1<s^{\prime}(x)q(x)<(\phi_{0})^{\ast}\quad \mbox{and}\quad 1<\alpha(x):=\frac{s(x)q(x)}{s(x)-q(x)}<(\phi_{0})^{\ast}, \quad \forall x\in \overline{\Omega}.
	\end{align}
	Thus,
	\begin{align*}
		\max_{x\in \overline{\Omega}}{s^{\prime}(x)q(x)}:=s^{\prime}(x_{0})q(x_0)<(\phi_{0})^{\ast}\quad \mbox{and}\quad  \max_{x\in \overline{\Omega}}{\alpha(x)}:=\alpha(x_0)<(\phi_{0})^{\ast}.
	\end{align*}
	Using Lemma \ref{min-max for Phi ast} and   $(H_5)$, we obtain
	 \begin{equation}\label{compact embedding 1}
	 	 \lim\limits_{t\to +\infty}\frac{|kt|^{s^{\prime}(x)q(x)}}{\Phi_{\ast}(x,t)}\leq \frac{k^{s^{\prime}(x_0)q(x_0)}}{\Phi_{\ast}(x,1)}\lim\limits_{t\to +\infty} \frac{1}{t^{(\phi_{0})^{\ast}-s^{\prime}(x_0)q(x_0)}}= 0 \; \mbox{ uniformly for } x\in \Omega.
	 \end{equation}
	 Using the same arguments above we show that
	 \begin{align}\label{ compact embedding alpha}
	 	 \lim\limits_{t\to +\infty}\frac{|kt|^{\alpha(x)}}{\Phi_{\ast}(x,t)}=0\; \mbox{ uniformly for } x\in \Omega.
	 \end{align}
	Hence,  (\ref{compact embedding 1}) and (\ref{ compact embedding alpha}) imply that  $|t|^{s^{\prime}(x)q(x)}\ll\Phi_{\ast}$ and $|t|^{\alpha(x)}\ll\Phi_{\ast}$ respectively.  Thus,  from Theorem \ref{compact embedding theorem1 } we have the following compact embeddings
	\begin{align}\label{ compact embedding of s'q}
	W_{0}^{1,\Phi}(\Omega)\hookrightarrow\hookrightarrow L^{s^{\prime}(x)q(x)}(\Omega),
	\end{align}
	and
	\begin{align}\label{ compact embedding of alpha 2}
	W_{0}^{1,\Phi}(\Omega) \hookrightarrow\hookrightarrow L^{\alpha(x)}(\Omega).
	\end{align}	
 Now, by using the H\"{o}lder   inequality,  Proposition \ref{proposition inequality}, and (\ref{ compact embedding of s'q}), we have for all $u$ in $ W_{0}^{1,\Phi}(\Omega)$
	\begin{align}\label{boundedness of V}
	|\mathcal{J}(u)|\leq c_0\|V\|_{s(x)}\||u|^{q(x)}\|_{s^{\prime}(x)}&\leq c_1\|V\|_{s(x)} \max\{\|u\|^{q_{-}}_{s^{\prime}(x)q(x)},\|u\|^{q_{+}}_{s^{\prime}(x)q(x)}\}\\ \nonumber
	&\leq c_2\|V\|_{s(x)} \max\{\|u\|^{q_{-}}_{1,\Phi},\|u\|^{q_{+}}_{1,\Phi}\}
	\end{align}
	where $c_{i}, i=0,1,2$ are positive constants. Hence, $\mathcal{J}$ is well defined.
	Moreover, since $q^{+}<(\phi_{0})^{\ast}$ then, as in the proof of  relation (\ref{compact embedding 1}), the space  $W_{0}^{1,\Phi}(\Omega)$ is compactly embedded in $L^{q(x)}(\Omega)$. Hence,   using (\ref{ compact embedding of alpha 2}) and   following the same arguments as in the proof of \cite[Proposition 2]{Kefi},  we obtain $\mathcal{J}\in \mathcal{C}^{1}(W_{0}^{1,\Phi}(\Omega),\mathbb{R})$ and
	$$\langle \mathcal{J}^{\prime}(u),v\rangle= \int_{\Omega}V(x)|u|^{q(x)-2}uvdx,\quad \forall u, v\in W_{0}^{1,\Phi}(\Omega).$$
	Finally, from the properties of $\Psi$,  $\Psi(x,k)$ is bounded for any positive constant $k$. Using Lemma \ref{min-max for Phi}  and the fact that  $\psi^{0}<(\phi_{0})^{\ast}$ we obtain for any $k>0$
	\begin{equation}\label{compact embedding of Phi in Psi}
	\lim\limits_{t\to +\infty}\frac{\Psi(x,kt)}{\Phi_{\ast}(x,t)}\leq \frac{\Psi(x,k)}{\Phi_{\ast}(x,1)}\lim\limits_{t\to +\infty} \frac{1}{t^{(\phi_{0})^{\ast}-\psi^{0}}}= 0 \; \mbox{ uniformly for } x\in \Omega.
	\end{equation}
	Hence,  $\Psi\ll\Phi_{\ast}$, which implies by Theorem \ref{compact embedding theorem1 } that
	\begin{align}\label{compact embedded of Psi}
	W_{0}^{1,\Phi}(\Omega) \hookrightarrow\hookrightarrow L^{\Psi}(\Omega).
	\end{align}
	 Consequently, from  $(\ref{growth of f})$,  the functional $\mathcal{F}$ is well defined and $\mathcal{F}\in \mathcal{C}^{1}(W_{0}^{1,\Phi}(\Omega),\mathbb{R})$ with its derivative  given by $$\langle\mathcal{F}'(u), v\rangle=\int_{\Omega}f(x,u)vdx,\quad \forall u, v\in W_{0}^{1,\Phi}(\Omega).$$
	 The proof of this  proposition is now complete.
\end{proof}
\begin{proposition}\label{S plus}
	~~
	\begin{itemize}
		\item[i)] The mapping $\mathcal{H}^{\prime} : W_{0}^{1,\Phi}(\Omega)\to (W_{0}^{1,\Phi}(\Omega))^{\ast}$ defined by
		\begin{align}
		\langle\mathcal{H}^{\prime}(u),v \rangle= \int_{\Omega}\phi(x,|\nabla u|)\nabla u \nabla vdx, \quad\forall u, v\in W_{0}^{1,\Phi}(\Omega),
		\end{align}
		is bounded, coercive, strictly monotone homeomorphism, and is of type $(S_{+})$,  namely,
		\begin{align*}
		u_{n}\rightharpoonup u \mbox{ in } W_{0}^{1,\Phi}(\Omega) \;\mbox{ and }\; \limsup_{n\to\infty}\langle \mathcal{H}^{\prime}(u_n), u_{n}-u\rangle\leq 0\; \mbox{ imply that }\; u_{n}\to u\; \mbox{ in } W_{0}^{1,\Phi}(\Omega),
		\end{align*}
		where $\rightharpoonup $ and $\to$  denote the weak and strong convergence in $W_{0}^{1,\Phi}(\Omega)$, respectively.
		\item[ii)] The functional $\mathcal{F}$ is sequentially weakly continuous, namely, $u_n \rightharpoonup u$ in $W_{0}^{1,\Phi}(\Omega)$ implies $\mathcal{F}(u_n)\to \mathcal{F}(u)$.  In addition, the mapping $\mathcal{F}' : W_{0}^{1,\Phi}(\Omega) \to (W_{0}^{1,\Phi}(\Omega))^{\ast} $ defined by
		\begin{align}
				\langle\mathcal{F}'(u), v\rangle=\int_{\Omega}f(x,u)vdx,\quad \forall u, v\in W_{0}^{1,\Phi}(\Omega),
		\end{align}
		 is a completely continuous linear operator.
	\end{itemize}
\end{proposition}
\begin{proof}
  We refer the reader to \cite[Theorem 2.2]{X. Fan2} for the proof of the first item    and to  \cite[Lemma 4.1]{P. Zhao} for that of the second one.
\end{proof}
We note that, by  Proposition \ref{the functional I} and  Definition \ref{definition of weak solution}, $u$ is a weak solution of problem $(P)$ if and only if $u$ is a critical point of the functional $\mathcal{I}$. Hence, we shall use  critical point theory tools to show our main results.

 To establish Theorem \ref{main theorem 1}  we will prove that the functional  $\mathcal{I}$ has a global minimum. 
\begin{proof}[Proof of Theorem \ref{main theorem 1}] Firstly, we show that $\mathcal{I}$ is coercive, namely, $\mathcal{I}(u)\to +\infty$ as $\|u\|_{1,\Phi}\to +\infty$. From $(\ref{growth of f})$, we have
	$$|F(x,t)|\leq C_{0}\Psi(x,t)+h(x)|t|, \quad \forall (x,t)\in \Omega\times \mathbb{R}.$$
Then, by applying Lemma \ref{min-max for Phi}, Poincaré	and H\"{o}lder's  inequalities, and  using similar arguments as in the proof of relation $(\ref{boundedness of V})$ we obtain
\begin{align*}
\mathcal{I}(u)&=\int_{\Omega}\Phi(x,|\nabla u|)dx+\int_{\Omega}\frac{V(x)}{q(x)}|u|^{q(x)}dx -\int_{\Omega}F(x,u)dx\\
&\geq \|u\|^{\phi_{0}}_{1,\Phi}-c_{1}\|V\|_{s(x)} \|u\|^{q_{+}}_{1,\Phi}-c_{2}\|u\|^{\psi^{0}}_{\Psi}-c_{3}\|h\|_{\tilde{\Psi}}\|u\|_{\Psi}.
\end{align*}
Using the  fact that $W_{0}^{1,\Phi}(\Omega)$ is compactly embedded in $L^{\Psi}(\Omega)$ (see the proof of relation (\ref{compact embedded of Psi})),  the previous inequality becomes
\begin{align*}
\mathcal{I}(u)\geq \|u\|^{\phi_{0}}_{1,\Phi}-c_{1}\|V\|_{s(x)} \|u\|^{q_{+}}_{1,\Phi}-c^{\prime}_{2}\|u\|^{\psi^{0}}_{1,\Phi}-c^{\prime}_{3}\|h\|_{\tilde{\Psi}}\|u\|_{1,\Phi}.
\end{align*}
  Since $1<q^{+}<\phi_{0}$ and $\psi^{0}<\phi_{0}$, we then have  $\mathcal{I}(u)\to +\infty$ as $\|u\|_{1,\Phi}\to +\infty$. To complete the proof we show that the functional $\mathcal{I}$ is weakly lower semi-continuous, namely, $u_n\rightharpoonup u$ in $W_{0}^{1,\Phi}(\Omega)$ implies $\mathcal{I}(u)\leq \liminf_{n\to \infty}\mathcal{I}(u_n)$.
 Suppose that $u_n \rightharpoonup u$ in $W_{0}^{1,\Phi}(\Omega)$. Since  the functional $\mathcal{H}\in \mathcal{C}^{1}(W_{0}^{1,\Phi}(\Omega),\mathbb{R})$ is  strictly convex ($\mathcal{H}' $ is strictly monotone), then we have $\mathcal{H}(u_n)>\mathcal{H}(u)+\langle \mathcal{H}^{\prime}(u),u_n -u \rangle$ which implies that   $\mathcal{H}$ is weakly lower semi-continuous on $W_{0}^{1,\Phi}(\Omega)$. Concerning the functional $\mathcal{J}$; since $W_{0}^{1,\Phi}(\Omega)$  is compactly embedded in $L^{s^{\prime}(x)q(x)}(\Omega)$ (see the proof of relation (\ref{ compact embedding of s'q})), then $u_n \to u$ in $L^{s^{\prime}(x)q(x)}(\Omega)$. This fact combined with relation (\ref{boundedness of V}) yields that $\mathcal{J}(u_n)\to \mathcal{J}(u)$. Finally, from Proposition \ref{S plus} -ii),  $\mathcal{F}$ is sequentially weakly continuous. Then,  we have $\mathcal{F}(u_n)\to \mathcal{F}(u)$.
\end{proof}
 We shall now prove Theorem \ref{main theorem 2} by using the Mountain Pass Theorem (see \cite{AR}). Since the proof of this theorem is quite long, we will divide it into several lemmas. Firstly, we show that the functional $\mathcal{I}$ has a geometrical structure. Secondly, we show that  $\mathcal{I}$ satisfies the Palais-Smale condition at level $\tilde{c}$ (see the Definition \ref{ Palais-Smale condition} below). To this end, we show that any Palais-Smale sequence   at the level $\tilde{c}$ for $\mathcal{I}$ (see the Definition \ref{ Palais-Smale sequence}) is  bounded in $ W_{0}^{1,\Phi}(\Omega)$,  and then has  a strongly convergent subsequence.

 Let us start by the following lemma.
 \begin{lemma}\label{estimation of Vu}
 	Assume that    $\phi^{0}<q^{-}\leq  q^{+}<(\phi_{0})^{\ast},$ $q^{+}-\frac{1}{2}\phi_{0}<q^{-}$ and $s(x)>\frac{q(x)(\phi_{0})^{\ast}}{(\phi_{0})^{\ast}-q(x)}$ for every $x\in \overline{\Omega}$. Then, for any function $V\in L^{s(x)}(\Omega)$ we have
 	\begin{equation}
 	\int_{\Omega}|V(x)||u|^{q(x)}dx\leq C\|V\|^{\frac{\alpha}{r^{-}}}_{s(x)} \left[  M_{1}+ M_{2}\left(\|u\|^{2(q^{+}-\theta)}_{1,\Phi}+\|u\|^{2\frac{\theta r^{+}}{r^{-}}}_{1,\Phi}\right) \right],
 	\end{equation}
 	 	where $\alpha =r^{+}$ if $\|V\|_{s(x)}>1$ and $\alpha=r^{-}$ if $\|V\|_{s(x)}\leq 1$, and $2\frac{\theta r^{+}}{r^{-}}< 2(q^{-}-\theta)<2(q^{+}-\theta)<\phi_{0}$ for  some  measurable function $r$ and  positive constants $C$ and $\theta$.
 \end{lemma}
 \begin{proof}
 	Since we have $q^{+}-\frac{1}{2}\phi_{0}<q^{-}$, then there exists $\theta>0$ such that $q^{+}-\frac{1}{2}\phi_{0}<\theta<q^{-}$.  This fact implies that $2(q^{-}-\theta)<2(q^{+}-\theta)<\phi_{0}$ and $1+\theta-q^{+}>0$. Let $r$ be any measurable function  satisfying,
 	\begin{align}
 	\max&\left\{\frac{s(x)}{1+\theta s(x)},\frac{(\phi_{0})^{\ast}}{(\phi_{0})^{\ast}+\theta-q(x)} \label{first prop of r} \right\}<r(x)<\min\left\{\frac{s(x)(\phi_{0})^{\ast}}{(\phi_{0})^{\ast}+\theta s(x)},\frac{1}{1+\theta-q(x)} \right\},\\
 	&\theta\left(\frac{r^{+}}{r^{-}}+1\right)<q^{-}, \qquad\forall x\in \Omega. \label{second prop of r}
 	\end{align}
 	It is clear that $r\in L^{\infty}(\Omega)$ and $1<r(x)<s(x)$. Now, by using H\"{o}lder's inequality, we get
 	\begin{align}\label{first calculation}
 	\int_{\Omega}|V(x)||u|^{q(x)}dx \leq C\|V|u|^{\theta}\|_{r(x)}\||u|^{q(x)-\theta}\|_{(r(x))^{\prime}}
 	\end{align}
 	Without loss of generality, we may assume that $\|V(x)|u|^{\theta}\|_{r(x)}>1$. Using again H\"{o}lder's inequality,  $(\ref{min and max})$, and  Proposition \ref{proposition inequality}, we obtain
 	\begin{align}\nonumber
 	\|V|u|^{\theta}\|_{r(x)}&\leq \left[ \int_{\Omega}|V(x)|^{r(x)}|u|^{\theta r(x)}dx\right]^{\frac{1}{r^{-}}}\\ \nonumber
 	&\leq C_{1}\||V|^{r(x)}\|^{\frac{1}{r^{-}}}_{\frac{s(x)}{r(x)}}   \||u|^{\theta r(x)}\|^{\frac{1}{r^{-}}}_{(\frac{s(x)}{r(x)})^{\prime}}\\ \label{last relation 1}
 	&\leq C_{2}\|V\|^{\frac{\alpha}{r^{-}}}_{s(x)}\left(1+\|u\|^{\frac{\theta r^{+}}{r^{-}}}_{\theta r(x)(\frac{s(x)}{r(x)})^{\prime}}\right) ,
 	\end{align}
 	where $\alpha =r^{+}$ if $\|V\|_{s(x)}>1$ and $\alpha=r^{-}$ if $\|V\|_{s(x)}\leq 1$. \\
 	Using the same arguments as above, we obtain
 	\begin{align}\label{last relation 2}
 	\||u|^{q(x)-\theta}\|_{(r(x))^{\prime}}\leq 1+\|u\|^{q^{+}-\theta}_{(q(x)-\theta)(r(x))^{\prime}}.
 	\end{align}
 	Since $r(x)$ is chosen such that $(\ref{first prop of r})$  is fulfilled then
 	\begin{align*}
 	1<\theta r(x)\left(\frac{s(x)}{r(x)}\right)^{\prime}<(\phi_{0})^{\ast}\quad \mbox{ and }\quad 1<(q(x)-\theta)(r(x))^{\prime}<(\phi_{0})^{\ast}, \forall x\in \Omega.
 	\end{align*}
 Since $\Phi_{\ast}$ satisfies $(H_5)$, then by	using Lemma \ref{min-max for Phi ast},  we have   $|t|^{\theta r(x)\left(\frac{s(x)}{r(x)}\right)^{\prime}}\preccurlyeq \Phi_{\ast}$ and $|t|^{(q(x)-\theta)(r(x))^{\prime}}\preccurlyeq \Phi_{\ast}$, which imply that $L^{\Phi_{\ast}}(\Omega)$ is continuously embedded in $ L^{\theta r(x)\left(\frac{s(x)}{r(x)}\right)^{\prime}}(\Omega)$ and in $L^{(q(x)-\theta)(r(x))^{\prime}}(\Omega)$. Therefore, from Theorem \ref{compact embedding theorem1 },   $W_{0}^{1,\Phi}(\Omega)$ is continuously embedded in $L^{\theta r(x)\left(\frac{s(x)}{r(x)}\right)^{\prime}}(\Omega)$ and in $ L^{(q(x)-\theta)(r(x))^{\prime}}(\Omega)$. Consequently, the relations (\ref{last relation 1}) and (\ref{last relation 2}) become respectively
 	\begin{align}\label{third ineq}
 	\|V|u|^{\theta}\|_{r(x)} \leq C^{\prime}\|V\|^{\frac{\alpha}{r^{-}}}_{s(x)}\left(1+\|u\|^{\frac{\theta r^{+}}{r^{-}}}_{1,\Phi}\right)
 	\end{align}
 	and
 	\begin{align}\label{fourth ineq}
 	\||u|^{q(x)-\theta}\|_{(r(x))^{\prime}}\leq C^{\prime\prime}\left(1+\|u\|^{q^{+}-\theta}_{1,\Phi}\right).
 	\end{align}
 	Substituting (\ref{third ineq}) and (\ref{fourth ineq}) into (\ref{first calculation}), and  using Young's inequality  we obtain
 	\begin{align}\label{V.u^{q(x)}}
 	\int_{\Omega}|V(x)||u|^{q(x)}dx \leq C\|V\|^{\frac{\alpha}{r^{-}}}_{s(x)}\left[ M_{1}+ M_{2}\left(\|u\|^{2(q^{+}-\theta)}_{1,\Phi}+\|u\|^{2\frac{\theta r^{+}}{r^{-}}}_{1,\Phi}\right)\right],
 	\end{align}
 	where $C, M_{1},$ and $M_{2}$ are positive constants.
 \end{proof}
 \begin{lemma}\label{geometrical structure}
 Assume that the assumptions $(\Phi),$ $(f^{\prime}_0),$   $(f_{2})$ and $(f_3)$ hold. Furthermore, assume that $\phi^{0}<\min\{\psi_{0},q^{-}\},$ $\max\{\psi^{0},q^{+}\}<(\phi_{0})^{\ast},$ $q^{+}-\frac{1}{2}\phi_{0}<q^{-},$ and $s(x)>\frac{q(x)(\phi_{0})^{\ast}}{(\phi_{0})^{\ast}-q(x)}$ for every $x\in\overline{\Omega}$. Then, the functional $\mathcal{I}$ has a geometrical structure, that is, $\mathcal{I}$ satisfies the following properties
 \begin{enumerate}
 	\item [(i)] there exist $\rho>0$ and $\beta>0$ such that    $\mathcal{I}(u)\geq \beta$ for any $u\in W_{0}^{1,\Phi}(\Omega)$ with $\|u\|_{1,\Phi}=\rho$.
 	\item [(ii)] there exists  $u_{0}\in W_{0}^{1,\Phi}(\Omega)$ such that   $\|u_{0}\|_{1,\Phi}>\rho$ and  $\mathcal{I}(u_{0})\leq 0$.
 \end{enumerate} 	
 \end{lemma}
 \begin{proof}
 	(i)  Firstly, from $(f^{\prime}_0)$ and $(f_3)$ it follows that, for all given $\epsilon>0$ there exists $C(\epsilon)>0$,
 	such that
 	\begin{equation}
 	|F(x,t)|\leq \epsilon \Phi(x,t)+C(\epsilon) \Psi(x,t), \quad \forall (x,t)\in \Omega\times \mathbb{R}.
 	\end{equation}
 Using   Lemma \ref{min-max for Phi},  the Poincaré   inequality, and the fact that  $W_{0}^{1,\Phi}(\Omega)$ is compactly embedded in $L^{\Psi}(\Omega)$, we obtain
 \begin{equation}\label{boundedness of F}
 \int_{\Omega} |F(x,t)|dx \leq \epsilon \max\{ \|u\|_{1,\Phi}^{\phi_{0}}, \|u\|^{\phi^{0}}_{1,\Phi}\}+C^{\prime}(\epsilon)\max\{ \|u\|_{1,\Phi}^{\psi_{0}}, \|u\|^{\psi^{0}}_{1,\Phi}\}.
 \end{equation}
 Using the same arguments  as in the proof of relation $(\ref{boundedness of V})$,   we obtain
 \begin{align}\label{boundedness of V 2}
 \int_{\Omega}\frac{V(x)}{q(x)}|u|^{q(x)} \leq   C\|V\|_{s(x)} \max\{\|u\|^{q_{-}}_{1,\Phi},\|u\|^{q_{+}}_{1,\Phi}\}.
 \end{align}
 Now, by using the definition of $\mathcal{I}$ in $(\ref{definition of I})$, Lemma \ref{min-max for Phi},  and the relations $(\ref{boundedness of F})$-$(\ref{boundedness of V 2})$, we get
 \begin{align*}
 \mathcal{I}(u)&=\int_{\Omega}\Phi(x,|\nabla u|)dx+\int_{\Omega}\frac{V(x)}{q(x)}|u|^{q(x)}dx -\int_{\Omega}F(x,u)dx\\
 &\geq \min\{ \|u\|_{1,\Phi}^{\phi_{0}}, \|u\|_{1,\Phi}^{\phi^{0}}\}-C\|V\|_{s(x)} \max\{\|u\|^{q_{-}}_{1,\Phi},\|u\|^{q_{+}}_{1,\Phi}\}\\
 &\qquad -\epsilon \max\{ \|u\|_{1,\Phi}^{\phi_{0}}, \|u\|^{\phi^{0}}_{1,\Phi}\}-C^{\prime}(\epsilon)\max\{ \|u\|_{1,\Phi}^{\psi_{0}}, \|u\|^{\psi^{0}}_{1,\Phi}\},
 \end{align*}
 which implies that, for all $u\in W_{0}^{1,\Phi}(\Omega)$ with $\|u\|_{1,\Phi}<1,$
 \begin{align} \label{last inequality 1}
 \mathcal{I}(u)&\geq   \|u\|_{1,\Phi}^{\phi^{0}}-C\|V\|_{s(x)}\|u\|^{q_{-}}_{1,\Phi} -\epsilon \nonumber \|u\|_{1,\Phi}^{\phi_{0}}-C^{\prime}(\epsilon) \|u\|_{1,\Phi}^{\psi_{0}}\\ \nonumber
 &\geq \frac{1}{2}\|u\|_{1,\Phi}^{\phi^{0}}-C\|V\|_{s(x)}\|u\|^{q_{-}}_{1,\Phi}-C^{\prime}(\epsilon) \|u\|_{1,\Phi}^{\psi_{0}}\\
 &=\|u\|_{1,\Phi}^{\phi^{0}}\left(\frac{1}{2}-C\|V\|_{s(x)}\|u\|^{(q_{-})-\phi^{0}}_{1,\Phi}-C^{\prime}(\epsilon) \|u\|_{1,\Phi}^{\psi_{0}-\phi^{0}}\right).
 \end{align}
 Since $(q_{-})-\phi^{0}>0$ and $\psi_{0}-\phi^{0}>0$, then from (\ref{last inequality 1})  we can choose $\beta>0$ and $\rho>0$ such that $\mathcal{I}(u)\geq \beta>0$ for any $u\in W_{0}^{1,\Phi}(\Omega)$ with $\|u\|_{1,\Phi}=\rho$.

 (ii) From $(f_2)$, it follows that for any $L>0$ there exists a constant $C_L :=C(L)> 0$ depending on L, such that
 \begin{equation}\label{C(L)}
 F(x,t)\geq L |t|^{\phi^{0}}- C_L,\quad \forall (x,t)\in \Omega \times \mathbb{R}.
 \end{equation}
  Let $w\in W_{0}^{1,\Phi}(\Omega) $  with $w>0$. We take $t>1$  large enough to ensure that $\|tw\|_{1,\Phi}>1$. Then from (\ref{C(L)}) and Lemmas \ref{min-max for Phi}, \ref{estimation of Vu},  we have
 \begin{align*}
  \mathcal{I}(tw)&=\int_{\Omega}\Phi(x,|t\nabla w|)dx+\int_{\Omega}\frac{V(x)}{q(x)}|tw|^{q(x)}dx -\int_{\Omega}F(x,tw)dx\\
    &\leq  t^{\phi^{0}}\|w\|^{\phi^{0}}_{1,\Phi}+ C\|V\|^{\frac{\alpha}{r^{-}}}_{s(x)} \left[  M_{1}+ M_{2}\left(t^{2(q^{+}-\theta)}\|w\|^{2(q^{+}-\theta)}_{1,\Phi}+t^{2\frac{\theta r^{+}}{r^{-}}}\|w\|^{2\frac{\theta r^{+}}{r^{-}}}_{1,\Phi}\right) \right]\\
    &\qquad\qquad-L t^{\phi^{0}}\int_{\Omega}|w|^{\phi^{0}}dx+C_{L}|\Omega|\\
    &=t^{\phi^{0}}\left( \|w\|^{\phi^{0}}_{1,\Phi} -L\int_{\Omega}|w|^{\phi^{0}}dx\right)\\
    &\qquad+C\|V\|^{\frac{\alpha}{r^{-}}}_{s(x)} \left[  M_{1}+ M_{2}\left(t^{2(q^{+}-\theta)}\|w\|^{2(q^{+}-\theta)}_{1,\Phi}+t^{2\frac{\theta r^{+}}{r^{-}}}\|w\|^{2\frac{\theta r^{+}}{r^{-}}}_{1,\Phi}\right) \right]+C_{L}|\Omega|.
 \end{align*}
 By choosing  $L>0$  such that $\|w\|^{\phi^{0}}_{1,\Phi} -L\int_{\Omega}|w|^{\phi^{0}}dx<0$ and the fact that  $2\frac{\theta r^{+}}{r^{-}}<2(q^{+}-\theta)<\phi_{0}$,  then we obtain $\mathcal{I}(tw)\to -\infty$ as $t\to +\infty$.  The proof of this lemma  is complete.
 \end{proof}
 \begin{remark}
 	Note that in the proof of the geometrical structure lemma we do not need any sign condition on the potential $V$.
 \end{remark}
Now, we define  the  level at $\tilde{c}$ as follows
$$\tilde{c}= \inf_{\gamma\in \Gamma} \max_{t\in [0,1]} \mathcal{I}(\gamma(t)),$$
where
$\Gamma =\{\gamma\in \mathcal{C}([0,1], W_{0}^{1,\Phi}(\Omega)): \gamma(0)=0,\; \gamma(1)=u_{0}\}$ is the set of continuous paths joining $0$ and $u_{0}$, where $u_{0} \in W_{0}^{1,\Phi}(\Omega)$ is defined  in the previous  lemma. Let us recall the standard definitions of Palais-Smale sequence at the level $\tilde{c}$ and Palais-Smale condition at the level $\tilde{c}$ for a functional  $\mathcal{I}\in \mathcal{C}^{1}(E,\mathbb{R})$, where $E$ is a Banach space.
\begin{definition}\label{ Palais-Smale sequence}
	Let $E$ be a Banach space with dual space $E^{\ast}$ and  $(u_{n})$ a sequence in $E$. We say that $(u_{n})$ is a Palais-Smale sequence at the level $\tilde{c}$ for  a functional $\mathcal{I}\in \mathcal{C}^{1}(E,\mathbb{R})$ if
	$$ \mathcal{I}(u_{n}) \to \tilde{c},\; \; \mbox{ and }\; \; \|\mathcal{I}^{\prime}(u_{n})\|_{E^{*}}\to 0.$$
\end{definition}
\begin{definition}\label{ Palais-Smale condition}
	We say that a functional $\mathcal{I}$ satisfies the Palais-Smale condition at the level $\tilde{c}$  if any  Palais-Smale sequence at the level $\tilde{c}$ for $\mathcal{I}$ possesses a convergent subsequence.
\end{definition}
We note that, by  Lemma \ref{geometrical structure}  the existence of a Palais-Smale sequence at the level $\tilde{c}$ for our functional $\mathcal{I}$  is ensured. This  can be observed directly from the proof given in \cite{AR}.

Now, in order to  prove that  the functional $\mathcal{I}$ satisfies the Palais-Smale condition, we shall first show that any  Palais-Smale sequence for $\mathcal{I}$ is bounded. To this end, we have the following lemma:
\begin{lemma}\label{boundedness of u_n}
	Assume that the assumptions $(\Phi)$ and $(f^{\prime}_0)$-$(f_3)$ hold. Furthermore,  assume that $\phi^{0}<\min\{\psi_{0},q^{-}\},$ $\max\{\psi^{0},q^{+}\}<(\phi_{0})^{\ast},$ $q^{+}-\frac{1}{2}\phi_{0}<q^{-},$  and $s(x)>\frac{q(x)(\phi_{0})^{\ast}}{(\phi_{0})^{\ast}-q(x)}$ for every $x\in \overline{\Omega}$. If $V$ has a constant sign  a.e. on $\Omega$, 	then any Palais-Samle sequence at the level $\tilde{c}$ for $\mathcal{I}$ is bounded in $W_{0}^{1,\Phi}(\Omega)$.
\end{lemma}
\begin{proof}
	Let $(u_{n})$ be a  Palais-Smale sequence at the level $\tilde{c}$ for $\mathcal{I}$ in $W_{0}^{1,\Phi}(\Omega)$. We prove by contradiction that $(u_{n})$ is bounded in $W_{0}^{1,\Phi}(\Omega)$. Assuming that $(u_{n})$ is unbounded in $W_{0}^{1,\Phi}(\Omega)$, that is, $\|u_{n}\|_{1,\Phi}\to +\infty$.
	
	Let  $v_{n}:= \frac{u_{n}}{\|u_{n}\|_{1,\Phi}}$. It is clear that   $(v_{n})$ is bounded in $W_{0}^{1,\Phi}(\Omega)$. Hence,  there exists  a subsequence denoted again   $(v_{n})$ such that $v_{n}$ converges weakly  to $v$ in $W_{0}^{1,\Phi}(\Omega)$. Since  $W_{0}^{1,\Phi}(\Omega)$ is compactly embedded in $ L^{\Psi}(\Omega)$ (see the proof of relation (\ref{compact embedded of Psi})), thus $v_{n}$ converges strongly to $v$    in  $L^{\Psi}(\Omega)$, and then  a.e. in $\Omega$.
	
	 Define $ \Omega_{\neq}:=\{x\in \Omega: |v(x)|\neq 0 \}$.  We consider two possible cases: $|\Omega_{\neq}|=0$ or $|\Omega_{\neq}|>0$.  Firstly,  we assume that $|\Omega_{\neq}|=0$, that is, $v=0$ a.e. in $\Omega$. From the definition of $\mathcal{I}$ in $(\ref{definition of I})$,  Lemma \ref{min-max for Phi}, and the fact that $\|u_{n}\|_{1,\Phi}\to +\infty$, we get
 \begin{align}\nonumber
 \|u_{n}\|^{\phi_{0}}_{1,\Phi}&\leq \mathcal{I}(u_n)- \int_{\Omega}\frac{V(x)}{q(x)}|u_n|^{q(x)}dx +\int_{\Omega}F(x,u_n)dx\\
 &\leq \mathcal{I}(u_n)+\frac{1}{q^{-}}\int_{\Omega}|V(x)||u_n|^{q(x)}dx +\int_{\Omega}F(x,u_n)dx, \label{the change 1}
 \end{align}
 which implies that
 \begin{equation}\label{ineq.2}
 1\leq \frac{\mathcal{I}(u_n)}{\|u_{n}\|^{\phi_{0}}_{1,\Phi}}+\frac{1}{q^{-}\|u_{n}\|^{\phi_{0}}_{1,\Phi}}\int_{\Omega}|V(x)||u_n|^{q(x)}dx +\int_{\Omega}\frac{F(x,u_n)}{\|u_{n}\|^{\phi_{0}}_{1,\Phi}}dx.
 \end{equation}
 Now, we shall show that all terms of the right-hand side of (\ref{ineq.2}) tend to zero when $n$ is large enough,  which is the desired contradiction. Since $(u_n)$ is a Palais-Smale sequence type, then $(\mathcal{I}(u_n))$ is bounded. Hence, the first term of the right-hand side of (\ref{ineq.2}) tends to zero as $n$ is large enough. For the second one, from Lemma \ref{estimation of Vu} we get
 \begin{align}\label{last inequ 3}
 \frac{1}{q^{-}\|u_{n}\|^{\phi_{0}}_{1,\Phi}}\int_{\Omega}|V(x)||u_n|^{q(x)}dx\leq C\|V\|^{\frac{\alpha}{r^{-}}}_{s(x)} \frac{ M_{1}+ M_{2}\left(\|u_n\|^{2(q^{+}-\theta)}_{1,\Phi}+\|u_n\|^{2\frac{\theta r^{+}}{r^{-}}}_{1,\Phi}\right)}{q^{-}\|u_{n}\|^{\phi_{0}}_{1,\Phi}}.
 \end{align}
 Since, $2\frac{\theta r^{+}}{r^{-}}< 2(q^{-}-\theta)<2(q^{+}-\theta)<\phi_{0}$, then passing to the limit in (\ref{last inequ 3}), we obtain
 \begin{align}\label{last inequ 4}
 \frac{1}{q^{-}\|u_{n}\|^{\phi_{0}}_{1,\Phi}}\int_{\Omega}|V(x)||u_n|^{q(x)}dx \to 0, \mbox{ as } n\to +\infty.
 \end{align}	Hence, the second term tends to zero as $n$ is large enough.
 	For the third term, on the one hand it follows from the definition of $F$ that
 	\begin{equation}\label{bounded of F div with norm}
 	\int_{\{ |u_n|\leq M\}}\frac{|F(x,u_n)|}{\|u_{n}\|^{\phi_{0}}_{1,\Phi}}dx \leq \frac{C(M)}{\|u_{n}\|^{\phi_{0}}_{1,\Phi}}, 
 	\end{equation}
 	where $C(M)$ is a positive constant depending on $M$ defined in (\ref{Malghass condition}).
 	On the other hand, by using H\"{o}lder's inequality, we get
 	\begin{align*}
 	\int_{\{ |u_n|> M\}}\frac{F(x,u_n)}{\|u_{n}\|^{\phi_{0}}_{1,\Phi}}dx&=\int_{\{ |u_n|> M\}}\frac{F(x,u_n)}{|u_{n}|^{\phi_{0}}} |v_{n}|^{\phi_{0}}dx\\
 	&\leq 2\left\|\frac{F(x,u_n)}{|u_{n}|^{\phi_{0}}} \chi_{\{|u_n|> M\}}\right\|_{\Gamma}\| |v_{n}|^{\phi_{0}} \chi_{\{|u_n|> M\}} \|_{\tilde{\Gamma}}.
 	\end{align*}
 	Without loss of generality, we may suppose that  $\left\|\frac{F(x,u_n)}{|u_{n}|^{\phi_{0}}} \chi_{\{|u_n|> M\}}\right\|_{\Gamma}> 1$. Then, from  Lemma \ref{min-max for Phi}, we get
 	\begin{equation*}
 	\left\|\frac{F(x,u_n)}{|u_{n}|^{\phi_{0}}} \chi_{\{|u_n|> M\}}\right\|_{\Gamma} \leq \left[\int_{\{ |u_n|> M\}} \Gamma\left(x,\frac{F(x,u_n)}{|u_{n}|^{\phi_{0}}}\right)dx\right]^{\frac{1}{\gamma_{0}}}.
 	\end{equation*}
 	Hence, it follows from (\ref{Malghass condition}) that,
 	\begin{equation}\label{C plus Cprime}
 	\left\|\frac{F(x,u_n)}{|u_{n}|^{\phi_{0}}} \chi_{\{|u_n|> M\}}\right\|_{\Gamma} \leq C \left[ \int_{\Omega} H(x,u_n)dx\right]^{\frac{1}{\gamma_{0}}}+ C^{\prime},
 	\end{equation}
 	where $C$ and $C^{\prime}$ are positive constants independent of $n$.
 	
 	 In the case where  $V\leq 0$  a.e. on $\Omega$, then from the definition of the functional $\mathcal{I}$ we get
 	 \begin{align} \label{First equality }
 	 \phi^{0}\mathcal{I}(u_n) - \langle\mathcal{I}^{\prime}(u_n),u_n \rangle&=\int_{\Omega} \left[\phi^{0}\Phi(x,|\nabla u_n|)-\phi(x,|\nabla u_n|)|\nabla u_n|^{2}\right] dx\\ \nonumber
 	 &+ \int_{\Omega}V(x)\left(\frac{\phi^{0}}{q(x)}-1\right)|u_n|^{q(x)}dx +\int_{\Omega}(f(x,u_n)u_n -\phi^{0}F(x,u_n))dx.
 	 \end{align}
 	
 	 From $(\ref{delta2})$ and the fact that $\phi^{0}<q^{-}\leq q(x)$, the first and the second terms of the right-hand side of $(\ref{First equality })$ are nonnegative. Hence, the relation (\ref{First equality }) becomes
 	 \begin{align} \label{First equality 2 }
 	 \phi^{0}\mathcal{I}(u_n) - \langle\mathcal{I}^{\prime}(u_n),u_n \rangle\geq \int_{\Omega}H(x,u_n)dx.
 	 \end{align}
 	  It follows from (\ref{First equality 2 }) that, $\int_{\Omega} H(x,u_n)dx\leq C$, for $n$ large enough.
 	
 	 Now,  in the case where  $V\geq 0$  a.e. on $\Omega$, then from the definition of the functional $\mathcal{I}$ we get
 	  \begin{align} \label{First equality 4 }
 	  q^{+}\mathcal{I}(u_n) - \langle\mathcal{I}^{\prime}(u_n),u_n \rangle&=\int_{\Omega} \left[q^{+}\Phi(x,|\nabla u_n|)-\phi(x,|\nabla u_n|)|\nabla u_n|^{2}\right] dx\\ \nonumber
 	  &+ \int_{\Omega}V(x)\left(\frac{q^{+}}{q(x)}-1\right)|u_n|^{q(x)}dx +\int_{\Omega}H(x,u_n)dx.
 	  \end{align}
 	  Since $q(x)\leq q^{+}$, then following the same arguments as for (\ref{First equality 2 }),  we have also $\int_{\Omega} H(x,u_n)dx\leq C$, for $n$ large enough.
 	   This fact combined with relation (\ref{C plus Cprime}) yields
 	\begin{align}\label{last inequ 5}
 	\left\|\frac{F(x,u_n)}{|u_{n}|^{\phi_{0}}} \chi_{\{|u_n|> M\}}\right\|_{\Gamma} \leq C,\; \mbox{ for }\; n\; \mbox{ large enough}.
 	\end{align}
	where $C$ is a positive constant independent of $n$. Now, it remains to show that $\| |v_{n}|^{\phi_{0}} \chi_{\{|u_n|> M\}} \|_{\tilde{\Gamma}}\to 0$ as $n\to +\infty$.  Let $K(x,t):=\tilde{\Gamma}(x,|t|^{\phi_{0}})$. Since $\phi_{0}>1$ and $\tilde{\Gamma}\in N(\Omega)$, then it is clear that $K\in N(\Omega)$. Moreover, since $\Gamma$ satisfies $(H_2)$ then $K$  verifies the assumption (2) of Theorem \ref{compact embedding theorem1 } and by Remark \ref{bounded of K}, $K(x,k)$ is bounded for each $k>0$. Using Lemmas \ref{min-max for Phi} and \ref{min-max for Phi ast}, we get
\begin{align*}
 \lim\limits_{t\to +\infty}\frac{K(x,kt)}{\Phi_{\ast}(x,t)}\leq \frac{K(x,k)}{\Phi_{\ast}(x,1)}\lim\limits_{t\to +\infty} \frac{1}{t^{\phi_{0}(\gamma_{0})^{\prime}-(\phi_{0})^{\ast}}},
\end{align*}
where $(\gamma_{0})^{\prime}=\frac{\gamma_{0}}{\gamma_{0}-1}$ is defined as in (\ref{gamma condition}). Since $\frac{N}{\phi_{0}}<\gamma_{0}$, then $\phi_{0}(\gamma_{0})^{\prime}<(\phi_{0})^{\ast}$. From the last inequality it follows that
\begin{align*}
\lim\limits_{t\to +\infty}\frac{K(x,kt)}{\Phi_{\ast}(x,t)}=0, \mbox{uniformly for } x\in \Omega.
\end{align*}
Thus, form Theorem \ref{compact embedding theorem1 }, $W_{0}^{1,\Phi}(\Omega)$ is compactly embedded in $L^{K}(\Omega)$, which implies that
\begin{align*}
\int_{\Omega} \tilde{\Gamma}(x,|v_n|^{\phi_{0}})dx \to 0, \mbox{ as } n\to +\infty.
\end{align*}
Consequently,
\begin{align}\label{last inequ 6}
\||v_{n}|^{\phi_{0}} \chi_{\{|u_n|> M\}} \|_{\tilde{\Gamma}}\to 0, \mbox{ as } n\to +\infty.
\end{align}
Hence, passing to the limit in  (\ref{ineq.2}) and using  (\ref{last inequ 4}),(\ref{last inequ 5}) and (\ref{last inequ 6}), we obtain a contradiction.

Secondly, we assume that $|\Omega_{\neq}|>0$. Then obviously, $|u_n|=|v_n|\|u_{n}\|_{1,\Phi}\to +\infty$ in $\Omega_{\neq}$. Hence, for some positive real $M$ we have $\Omega_{\neq}\subset \{x\in \Omega: |u_n|\geq M \}$ for $n$ large enough.
Using Lemma \ref{min-max for Phi}, we get
\begin{align*}
\frac{\mathcal{I}(u_n)}{\|u_n\|^{\phi^{0}}_{1,\Phi}}&\leq 1+\frac{1}{q^{-}\|u_{n}\|^{\phi^{0}}_{1,\Phi}}\int_{\Omega}|V(x)||u_n|^{q(x)}dx -\int_{\Omega}\frac{F(x,u_n)}{\|u_{n}\|^{\phi^{0}}_{1,\Phi}}dx\\
&=1+\frac{1}{q^{-}\|u_{n}\|^{\phi^{0}}_{1,\Phi}}\int_{\Omega}|V(x)||u_n|^{q(x)}dx -\int_{\{ |u_n|\leq M\}}\frac{|F(x,u_n)|}{\|u_{n}\|^{\phi^{0}}_{1,\Phi}}dx\\
&\qquad \qquad -\int_{\{ |u_n|> M\}}\frac{F(x,u_n)}{|u_{n}|^{\phi^{0}}} |v_{n}|^{\phi^{0}}dx.
\end{align*}
Now, using relations (\ref{last inequ 4}), (\ref{bounded of F div with norm}), assumption $(f_2)$ and Fatou's Lemma, we obtain a contradiction. Hence,  $(u_n)$ is bounded in $W_{0}^{1,\Phi}(\Omega)$. The proof of this lemma is complete.
\end{proof}
\begin{remark}
  The preceding lemma holds true under a slightly weaker assumption than  $V$ has a constant sign. Indeed,  assume that there exists a constant $\rho$ such that $\phi^{0}\leq\rho\leq q^{+}$ and $V(x)\left(\frac{\rho}{q(x)}-1\right)\geq 0$ a.e. on $\Omega$. Then, by taking $H(x,t)=f(x,t)t-\rho F(x,t)$ and following the same arguments as in (\ref{First equality })-(\ref{First equality 2 }), we obtain the previous lemma.
\end{remark}
To finish the proof of the Palais-Smale condition for $\mathcal{I}$, we only need to show the following lemma:
\begin{lemma}\label{Strongly converge}
	Assume that the assumptions of   Lemma \ref{boundedness of u_n} hold. Then, the Palais-Smale sequence at the level $\tilde{c}$ for $\mathcal{I}$ possesses a convergent subsequence.	
\end{lemma}	
\begin{proof}
		Let $(u_{n})$ be a  Palais-Smale sequence at the level $\tilde{c}$ for $\mathcal{I}$ in $W_{0}^{1,\Phi}(\Omega)$. Then, $\mathcal{I}^{\prime}(u_{n})\to 0$ in $(W_{0}^{1,\Phi}(\Omega))^{\ast}$ and from Lemma \ref{boundedness of u_n}, $(u_{n})$ is bounded in $W_{0}^{1,\Phi}(\Omega)$. As $W_{0}^{1,\Phi}(\Omega)$ is reflexive, then there exists  a subsequence denoted again   $(u_{n})$ such that $u_{n}$ converges weakly  to $u$ in $W_{0}^{1,\Phi}(\Omega)$.
		 From  Proposition \ref{S plus}-(i), the mapping $\mathcal{H}^{\prime}$ is of type $(S_{+})$. Thus,   to conclude the result of this lemma it suffices to show that
		\begin{equation}\label{ S_{+} of H}
		\limsup_{n\to\infty}\langle \mathcal{H}^{\prime}(u_n), u_{n}-u\rangle\leq 0.
		\end{equation}
	 Indeed, using the definition of $\mathcal{I}^{\prime}$ in  Proposition \ref{derivative of I}, we have
		\begin{align}\label{S_{+}}
\langle\mathcal{H}^{\prime}(u_n), u_n -u \rangle= 	\langle\mathcal{I}^{\prime}(u_{n}), u_n -u\rangle + \langle\mathcal{F}^{\prime}(u_{n}) , u_n -u\rangle -\langle\mathcal{J}^{\prime}(u_{n}) , u_n -u\rangle.
		\end{align}
	It is clear that,
	\begin{align}\label{I^{prime} tends to zero}
			\langle \mathcal{I}^{\prime}(u_{n}), u_n -u\rangle \to 0.
	\end{align}
	From  Proposition \ref{S plus}-(ii),  $\mathcal{F}^{\prime}$  is a completely continuous linear operator. Hence,  
	\begin{align}\label{ F tends to 0}
	 \langle\mathcal{F}^{\prime}(u_{n}), u_n -u\rangle \to 0.
	 \end{align}
	 Now, it remains to show that $\langle\mathcal{J}^{\prime}(u_{n}), u_n -u\rangle \to 0,$ that is,
	 \begin{align}\label{stong convergence of J}
	 \int_{\Omega} V(x)|u_n|^{q(x)-2}u_n (u_n-u)dx \to 0.
	 \end{align}
	 From the assumptions, we have $1<q(x)<(\phi_{0})^{\ast}$ and $1<\alpha(x)< (\phi_{0})^{\ast}$ for every $x\in \overline{\Omega}$ with $\alpha(x):=\frac{s(x)q(x)}{s(x)-q(x)}$. Then, as in the proof of relation (\ref{ compact embedding of alpha 2}), the space $W_{0}^{1,\Phi}(\Omega)$ is compactly embedded in $L^{q(x)}(\Omega)$ and in  $L^{\alpha(x)}(\Omega)$.  Since $(u_n)$ is bounded in $W_{0}^{1,\Phi}(\Omega)$, then $u_n$ converges strongly to $u$ in $L^{\alpha(x)}(\Omega)$. Consequently,  using H\"{o}lder's inequality  and   Proposition \ref{proposition inequality},  then  (\ref{stong convergence of J}) holds by the following inequality
	 \begin{align*}
	 |\int_{\Omega} V(x)|u_n|^{q(x)-2}u_n (u_n-u)dx|&\leq C_{0}\|V\|_{s(x)}\||u_n|^{q(x)-1}\|_{q^{\prime}(x)}\|u_n-u\|_{\alpha(x)}\\
	 &\leq C_{1}\|V\|_{s(x)}\|u_n\|_{q(x)}^{\tau} \|u_n-u\|_{\alpha(x)},
	 \end{align*}
	 where $C_1 $ is a positive constant independent of $n$ and $\tau \in \{q^{-}-1,q^{+}-1\}$. Finally, it follows from (\ref{I^{prime} tends to zero}), (\ref{ F tends to 0}) and (\ref{stong convergence of J}) that $(\ref{ S_{+} of H})$ holds. Hence, since $\mathcal{H}^{\prime}$ is of type $(S_{+})$, then $u_{n}$ converges strongly to $u$ in $W_{0}^{1,\Phi}(\Omega)$. The proof of Theorem \ref{main theorem 2}   is complete.
\end{proof}
Next, as $W_{0}^{1,\Phi}(\Omega)$ is a reflexive and separable Banach space, there exist $(e_j)_{j\in \mathbb{N^{\ast}}}\subseteq W_{0}^{1,\Phi}(\Omega)$ and $(e^{\ast}_j)_{j\in \mathbb{N^{\ast}}}\subseteq (W_{0}^{1,\Phi}(\Omega))^{\ast}$ such that
$$W_{0}^{1,\Phi}(\Omega)=\overline{\mbox{span}\{ e_j : j \in \mathbb{N^{\ast}} \}}, \qquad (W_{0}^{1,\Phi}(\Omega))^{\ast}=\overline{\mbox{span}\{ e^{\ast}_j : j \in \mathbb{N^{\ast}} \}}  $$
and
$$\langle e_i,e^{\ast}_{j} \rangle=\begin{cases}
1, & i=j \\
0, & i\neq j.
\end{cases}$$
For $k\in \mathbb{N^{\ast}}$ denote by
$$X_j = \mbox{span}\{e_j \}, \quad Y_k=\oplus^{k}_{j=1}X_{j}, \quad \mbox{ and } \quad  Z_{k}=\overline{\oplus^{\infty}_{j=k}X_{j}}.$$
\begin{proof}[Proof of Theorem \ref{main theorem 3, Fountain}]
We denote by $$\beta_{k}:=\sup\left\{\int_{\Omega} \Psi(x,|u|)dx : \|u\|_{1,\Phi}=1, u\in Z_k \right\}.$$ Since $\Psi \ll \Phi_{\ast}$, then $\lim\limits_{k\to +\infty}\beta_{k}=0$ (see \cite[Lemma 4.3]{P. Zhao}).	Now, we verify the conditions of Fountain theorem. It follows from   assumption $(f_4)$ that $\mathcal{F}$ is even, hence the functional $\mathcal{I}$ is  even.  From Lemmas \ref{boundedness of u_n} and \ref{Strongly converge}, $\mathcal{I}$ satisfies the Palais-Smale condition; hence the condition (3)  of Fountain theorem holds. It remains to prove that conditions $(1)$ and $(2)$ in Fountain  theorem hold.
	
	$(1)$ By $(f^{\prime}_{0})$, it follows that
	\begin{align}
	|F(x,t)|\leq C(\Psi(x,t)+|t|),\quad \forall (x,t)\in \Omega\times \mathbb{R}.
	\end{align}
	 Let $u\in Z_{k}$ with $\|u\|_{1,\Phi}>1$.  From the definition of $\mathcal{I}$ in (\ref{definition of I}),  Lemmas \ref{min-max for Phi}, \ref{estimation of Vu} and Poincaré's inequality,  we obtain
	\begin{align}\label{Fout. 1}\nonumber
	\mathcal{I}(u)&=\int_{\Omega}\Phi(x,|\nabla u|)dx+\int_{\Omega}\frac{V(x)}{q(x)}|u|^{q(x)}dx -\int_{\Omega}F(x,u)dx\\
	&\geq \|u\|^{\phi_{0}}_{1,\Phi}-C\|V\|^{\frac{\alpha}{r^{-}}}_{s(x)}\left[ M_{1}+ M_{2}\left(\|u\|^{2(q^{+}-\theta)}_{1,\Phi}+\|u\|^{2\frac{\theta r^{+}}{r^{-}}}_{1,\Phi}\right)\right]-C_{1}\int_{\Omega} \Psi(x,|u|)dx-C_{2}\|u\|_{1,\Phi},
	\end{align}
	where we recall that $\alpha =r^{+}$ if $\|V\|_{s(x)}>1$ and $\alpha=r^{-}$ if $\|V\|_{s(x)}\leq 1$. Furthermore,  from Lemma \ref{min-max for Phi}, we have
$$\int_{\Omega} \Psi(x,|u|)dx=\int_{\Omega} \Psi\left(x,\|u\|_{1,\Phi} \frac{|u|}{\|u\|_{1,\Phi}}\right)dx\leq \|u\|^{\psi^{0}}_{1,\Phi}  \int_{\Omega} \Psi\left(x, \frac{|u|}{\|u\|_{1,\Phi}}\right)dx.$$
Using the definition of $\beta_{k}$, the relation (\ref{Fout. 1})  becomes
\begin{align*}
\mathcal{I}(u)\geq \|u\|^{\phi_{0}}_{1,\Phi}-C\|V\|^{\frac{\alpha}{r^{-}}}_{s(x)}\left[ M_{1}+ M_{2}\left(\|u\|^{2(q^{+}-\theta)}_{1,\Phi}+\|u\|^{2\frac{\theta r^{+}}{r^{-}}}_{1,\Phi}\right)\right] -C_1 \|u\|^{\psi^{0}}_{1,\Phi} \beta_{k} -C_{2}\|u\|_{1,\Phi}.
\end{align*}
 	Now, let $u_k\in Z_k$ with $\|u\|_{1,\Phi}=r_k = (2C_1 \beta_k)^{\frac{1}{\phi_{0}-\psi^{0}}}$. Since $\phi_0 < \psi^{0}$ and $\lim\limits_{k\to +\infty}\beta_{k}=0$, then $r_k \to +\infty$ as $k \to +\infty$. Thus, we have
 	\begin{align*}
 	\mathcal{I}(u)\geq (2C_1 \beta_k)^{\frac{\phi_{0}}{\phi_{0}-\psi^{0}}}-C& \|V\|^{\frac{\alpha}{r^{-}}}_{s(x)}\left[ M_{1}+ M_{2}\left( {r_k}^{2(q^{+}-\theta)}+ {r_k}^{ 2\frac{\theta r^{+}}{r^{-}}}\right)\right]\\
 	&-C_{1}(2C_1 \beta_k)^{\frac{\psi^{0}}{\phi_{0}-\psi^{0}}}\beta_{k}-C_{2}r_k,
 	\end{align*}
 	\begin{align*}
 	\mathcal{I}(u)\geq \frac{1}{2}{r_k}^{\phi_{0}}-C\|V\|^{\frac{\alpha}{r^{-}}}_{s(x)}&\left[ M_{1}+ M_{2}\left( {r_k}^{2(q^{+}-\theta)}+ {r_k}^{ 2\frac{\theta r^{+}}{r^{-}}}\right)\right] -C_{2}r_k.
 	\end{align*}
 Since	$2\frac{\theta r^{+}}{r^{-}}< 2(q^{-}-\theta)<2(q^{+}-\theta)<\phi_{0}$ and $1<\phi_0$, then
 \begin{align*}
\inf_{\{u\in Z_{k}, \|u\|=r_k \}}\mathcal{I}(u) \to +\infty \mbox{ as }  k\to +\infty.
 \end{align*}

 (2)  Let $w\in Y_k $  with $w>0, \|w\|_{1,\Phi}=1$ and $t>1$. Then, from Lemmas \ref{min-max for Phi} and (\ref{C(L)}) we have
 \begin{align*}
 \mathcal{I}(tw)&\leq t^{\phi^{0}}\left( \|w\|^{\phi^{0}}_{1,\Phi} -L\int_{\Omega}|w|^{\phi^{0}}dx\right)\\
 &\qquad+C\|V\|^{\frac{\alpha}{r^{-}}}_{s(x)} \left[  M_{1}+ M_{2}\left(t^{2(q^{+}-\theta)}\|w\|^{2(q^{+}-\theta)}_{1,\Phi}+t^{2\frac{\theta r^{+}}{r^{-}}}\|w\|^{2\frac{\theta r^{+}}{r^{-}}}_{1,\Phi}\right) \right]+C_{L}|\Omega|.
 \end{align*}
 It is clear that we can choose $L>0$ so that  $\|w\|^{\phi^{0}}_{1,\Phi} -L\int_{\Omega}|w|^{\phi^{0}}dx<0$. With this fact and  since $2\frac{\theta r^{+}}{r^{-}}<2(q^{+}-\theta)<\phi_{0}$ then we have $\mathcal{I}(tw)\to -\infty$ as $t\to +\infty$. Thus, there exists $\tilde{t}>r_{k}>1$ such that $\mathcal{I}(\tilde{t}w)<0$. By setting $\rho_{k}=\tilde{t}$,  then we obtain
 $$\max_{\{u\in Y_{k}, \|u\|=\rho_k \}}\mathcal{I}(u) \leq 0.$$
 The proof of this theorem is complete.
 	\end{proof}
 	\section{Application}
 	Let us give an example of  function $f$ satisfying the   assumptions $(f^{\prime}_0)$-$(f_3)$ and for which our main Theorems \ref{main theorem 2} and \ref{main theorem 3, Fountain} hold.
 	
 	 Let us fix $\Phi(x,t)=\frac{1}{p(x)}|t|^{p(x)}$ with $p\in \mathcal{C}^{1-0}(\overline{\Omega})$. Then, the operator $\mbox{div}(\phi(x,|\nabla u|)\nabla u)$ involved in  $(P)$ is the $p(x)$-Laplacian operator, i.e.  $\Delta_{p(x)}u:=\mbox{div}(|\nabla u|^{p(x)-2}\nabla u)$. In this case, we have $\phi_{0}=p^{-}$ and $\phi^{0}=p^{+}$ with $1<p^{-}\leq p(x)\leq p^{+}<N$.
 	
 	 In the case where $V\geq 0$ a.e. on $\Omega$, we take   $F(x,t)=|t|^{q^{+}}\ln(1+|t|)$, with $q^{+}+1<\frac{Np^{-}}{N-p^{-}}$. The derivative with respect to $t$ of $F(x,t)$ is given by $F^{\prime}(x,t):=f(x,t)=q^{+}|t|^{q^{+} -2}t\ln(1+|t|)+\frac{t|t|^{q^{+} -1}}{1+|t|}$ and we have: $f(x,t)t-q^{+}F(x,t)=\frac{|t|^{q^{+}}}{1+|t|}.$
 	   It is clear that $f$ satisfies the assumptions $(f^{\prime}_{0}),$ $(f_{2}$)-$(f_{4})$. Moreover, since $\frac{F(x,t)}{|t|^{\theta}} \to 0$ for all $\theta>q^{+}$ then   from (\ref{consequence of AR}), $f$ does not satisfy the (A-R) condition. Now, it remains to show that the assumption $(f_{1})$ holds. To this end, let us consider the function $\Gamma(x,t)=|t|^{\beta}$, where $1<\frac{N}{p^{-}}<\beta<\frac{q^{+}}{q^{+}-p^{-}}$  . Then, $\Gamma\left(x,\frac{F(x,t)}{|t|^{p^{-}}}\right)=|t|^{\beta(q^{+}-p^{-})}\ln^{\beta}(1+|t|)$. Since $\beta(q^{+}-p^{-})<q^{+}$, then $\frac{|t|^{\beta(q^{+}-p^{-})+1}\ln^{\beta}(1+|t|)}{|t|^{q^{+}+1}}\to 0$ as $|t|\to +\infty$. Hence,  the assumption $(f_{1})$ holds.
 	
 	   Now, in the case where  $V\leq  0$ a.e. on $\Omega$ we can take $F(x,t)=|t|^{p^{+}}\ln(1+|t|)$. By the same arguments above,  the choice of $\Gamma(x,t)=|t|^{\beta}$, where $1<\frac{N}{p^{-}}<\beta<\frac{p^{+}}{p^{+}-p^{-}}$, ensures easily that $f$ verifies the  assumptions $(f^{\prime}_{0})$-$(f_4)$. Consequently, the main Theorems \ref{main theorem 2} and \ref{main theorem 3, Fountain} hold.
 	   \begin{remark}
 	   	~~~
 	   	\begin{enumerate}
 	   		\item In the case where $V\leq 0$ a.e. on $\Omega$, we can not take the same  function $F$ considered in the first case, i.e. $F(x,t)=|t|^{q^{+}}\ln(1+|t|)$. Indeed,   in this case,  the nonlinearity $f$ satisfies the (A-R) condition.
 	   		\item As in the first remark, we can not consider the function $F(x,t)=|t|^{p^{+}}\ln(1+|t|)$ when $V\geq 0$ a.e. on $\Omega$. Indeed, in this case    we have
 	   		$$f(x,t)t-q^{+} F(x,t)=(p^{+}-q^{+})|t|^{p^{+}}\ln(1+|t|)+\frac{|t|^{p^{+}+1}}{1+|t|} < 0, \mbox{ for } |t| \mbox{ large enough}.$$
 	   		Hence,   the nonlinearity $f$ do not satisfy the assumption $(f_{1})$.
 	   	\end{enumerate}
 	   \end{remark}

\end{document}